\documentclass[11pt]{article}
\usepackage[latin1]{inputenc}
\usepackage{amsmath,amsthm,amssymb,amsfonts,amscd}
\usepackage{latexsym}
\usepackage{graphicx}
\usepackage{mathrsfs}
\usepackage{cite}
\usepackage[title]{appendix}

\usepackage{enumitem}
\usepackage[colorlinks=true,urlcolor=blue,
citecolor=red,linkcolor=blue,linktocpage,pdfpagelabels,
bookmarksnumbered,bookmarksopen]{hyperref}

\makeatletter \@addtoreset{equation}{section} \makeatother

\newtheorem{theorem}{Theorem}[section]

\newtheorem{proposition}{Proposition}[section]
\newtheorem{lemma}{Lemma}[section]
\newtheorem{remark}{Remark}[section]

\allowdisplaybreaks

\begin{document}
\title{{Optimal Stability Bounds, Minimizers, and Critical-Point Stability for a Critical Nonlocal Sobolev Inequality on the Heisenberg Group}}

\author{Wenjing Chen\footnote{E-mail address:\, {\tt wjchen@swu.edu.cn} (W. Chen), {\tt zxwangmath@163.com} (Z. Wang).}\ \ and Zexi Wang\footnote{Corresponding author.} \\
\footnotesize School of Mathematics and Statistics, Southwest University,
Chongqing, 400715, P.R. China}

\date{ }
\maketitle

\begin{abstract}
\begingroup
We investigate the optimal Bianchi--Egnell-type quantitative stability constant for the critical nonlocal Sobolev inequality on the Heisenberg group $\mathbb{H}^{n}$,
\begin{equation*}
 S_{HL}(Q,\mu)\left(\int_{\mathbb{H}^{n}}\int_{\mathbb{H}^{n}}
 \frac{|u(\xi)|^{Q^{\ast}_{\mu}}|u(\eta)|^{Q^{\ast}_{\mu}}}
 {|\eta^{-1}\xi|^{\mu}}\,d\xi d\eta\right)^{\frac{1}{Q^{\ast}_{\mu}}}
 \leq \int_{\mathbb{H}^{n}}|\nabla_{\mathbb{H}}u|^{2}d\xi,
 \qquad u\in S^{1,2}(\mathbb{H}^{n}),
\end{equation*}
where $Q=2n+2$, $n\ge1$, $0<\mu<Q$, and $Q^{\ast}_{\mu}=(2Q-\mu)/(Q-2)$. Let $\mathfrak M$ be the manifold of Jerison--Lee bubbles and let $H_{NS}$ denote the optimal deficit-to-distance constant. After the Cayley transform, the second variation is governed by two coupled spectral forms on complex spherical harmonics. Their simultaneous diagonalization yields an explicit local spectral threshold, while a two-bubble expansion gives a second threshold associated with dichotomy.

To obtain strictness at the local threshold, we introduce a second-order normal correction and expand the nonlocal functional to fourth order. This produces an explicit coefficient $\Gamma_{n,\mu}$ such that $\Gamma_{n,\mu}<0$ implies $H_{NS}<H_{NS}^{\mathrm{spec}}$; in particular, an explicit algebraic calculation gives $\Gamma_{n,2}<0$ for every $n\ge1$. Together with concentration--compactness, the strict spectral and two-bubble bounds yield attainment of $H_{NS}$ when $0<\mu\le4$ and $\Gamma_{n,\mu}<0$; a separate comparison with the local Sobolev stability threshold gives attainment in the additional large-$\mu$ range.
The minimizer gives the strict comparison $H_{NS}>H_{BE}$ with the optimal constant for the local Folland--Stein--Sobolev inequality. We also prove that the sharp universal upper defect-to-distance constant is $1$. Finally, for the associated Euler--Lagrange equation, we introduce the optimal single-bubble residual quotient and establish the local spectral estimate $H_{CP}(1)\le H_{NS}^{\mathrm{spec}}$.
\endgroup



\end{abstract}

\tableofcontents

\section{Introduction}

{We begin with the Euclidean model that motivates the stability problem. For $N\geq3$, the sharp Sobolev inequality states}
\begin{equation}\label{Sobolev inequality}
\|\nabla u\|_{L^{2}(\mathbb{R}^{N})}^{2} \geq \mathcal{S}(N)\|u\|_{L^{2^{\ast}}(\mathbb{R}^{N})}^{2},\quad \forall \, u\in{D}^{1,2}(\mathbb{R}^{N}),
 \end{equation}
where ${\mathcal{S}(N)}=\pi N(N-2)\Big(\frac{\Gamma(\frac{N}{2})}{\Gamma(N)}\Big)^{\frac{2}{N}}$, $2^{\ast}=\frac{2N}{N-2}$, and $D^{1,2}(\mathbb{R}^N)$ denotes the completion of $C_c^\infty(\mathbb{R}^{N})$ under the norm $\|\nabla u\|_{L^{2}(\mathbb{R}^{N})}$.
{Aubin \cite{A76} and Talenti \cite{T76}, using variational and rearrangement arguments, showed that equality in \eqref{Sobolev inequality} holds precisely for the Aubin-Talenti bubbles}
\begin{equation*}
\mathcal{U}_{\lambda,a}(x)=\lambda^{\frac{N-2}{2}}\mathcal{U}(\lambda(x-a)),~~\lambda>0,~a\in\mathbb{R}^{N},
\end{equation*}
{where $\mathcal{U}(x)= [N(N-2)]^{\frac{N-2}{4}}(1+|x|^{2})^{-\frac{N-2}{2}}$ is, up to translations and scalings, the unique positive solution of the Euler--Lagrange equation associated with \eqref{Sobolev inequality}:}
\begin{equation*}
 -\Delta u=|u|^{2^{\ast}-2}u,\qquad \mathrm{in}~\mathbb{R}^{N}.
\end{equation*}

{Let $\mathcal{M}=\{c\,\mathcal{U}_{\lambda,a}:c\in\mathbb{R},\lambda>0,a\in\mathbb{R}^{N}\}$ be the manifold of Talenti bubbles. Quantitative stability asks whether the Sobolev deficit controls the squared distance to $\mathcal{M}$. This question was raised by Brezis and Lieb \cite{BL85}, who asked whether an estimate of the following form holds for a suitable distance $\omega$:}
\begin{equation*}
 \|\nabla u\|_{L^{2}(\mathbb{R}^{N})}^{2}-\mathcal{S}(N)\|u\|_{L^{2^{\ast}}(\mathbb{R}^{N})}^{2}\geq c \ \omega(u,\mathcal{M})^2,
 \quad \forall \, u\in{D}^{1,2}(\mathbb{R}^{N}).
\end{equation*}
{Bianchi and Egnell \cite{BE91} answered this question affirmatively by combining local spectral analysis with a global compactness argument:}
\begin{equation*}
 \|\nabla u\|_{L^{2}(\mathbb{R}^{N})}^{2}-\mathcal{S}(N)\|u\|_{L^{2^{\ast}}(\mathbb{R}^{N})}^{2}\geq c_{BE} ~ \mathrm{dist}(u,\mathcal{M})^2,
 \quad \forall \, u\in{D}^{1,2}(\mathbb{R}^{N}),
\end{equation*}
{here $c_{BE}>0$ is the optimal stability constant and $\mathrm{dist}(u,\mathcal{M})=\inf_{c,\lambda,a}\|u-c\, \mathcal{U}_{\lambda,a}\|_{D^{1,2}}$. The same paradigm has since been developed for second-order Sobolev inequalities \cite{LW2000}, fractional Sobolev inequalities \cite{CFW13}, and $p$-Laplacian problems \cite{FN19,FZ22,N20}. Local expansion around a bubble gives the spectral upper bound}
\begin{equation*}
 0<c_{BE}\leq c_{BE}^{\mathrm{spec}}:=1-\frac{\nu_1}{\nu_2}=\frac{4}{N+4}.
\end{equation*}
Here, $\nu_1$ and $\nu_2$ denote the second and third eigenvalues of the eigenvalue problem
\begin{equation*}
 -\Delta \varphi=\nu \mathcal{U}^{2^{\ast}-2}\varphi,\quad\mathrm{in}~\mathbb{R}^{N},~~\varphi\in D^{1,2}(\mathbb{R}^{N})\setminus \{0\},
\end{equation*}
whose eigenvalues satisfy the ordering $\nu_0=1<\nu_1=2^{\ast}-1<\nu_2=\frac{(N+2)(N+4)}{N(N-2)}<\cdots<\nu_k\rightarrow+\infty$ as $k\rightarrow+\infty$.
{On the lower-bound side, Dolbeault et al. \cite{DEFF25} obtained the first explicit dimension-dependent estimate}
$ c_{BE}\geq \frac{\beta}{N}$
for some explicitly given constant $\beta>0$. Related results concerning higher-order and fractional-order counterparts can be found in \cite{CLT25,CLT24,CLT24'}.
{On the upper-bound and compactness side, a refined expansion of the Bianchi--Egnell quotient along a carefully chosen perturbation allowed}
K\"{o}nig \cite{K23} {to prove that $c_{BE}<c_{BE}^{\mathrm{spec}}$. He subsequently established} \cite{K25}
\begin{equation*}
 c_{BE}< c_{BE}^{2\text{-peak}}:=1-2^\frac{N-2}{N},
\end{equation*}
{and proved that the infimum defining $c_{BE}$ is attained. Related stability-minimizer results include} the Caffarelli--Kohn--Nirenberg inequality \cite{DT23,DTW25,WW24}, the Hardy-Sobolev inequality \cite{CGK25}, and the fractional Sobolev trace inequality \cite{ZZZ25}.

{The nonlocal analogue is generated by the Hardy--Littlewood--Sobolev (HLS) inequality \cite{HL28,L83,LL01}. Suppose that} $N\geq 1$, $0<\mu< N$ and $t,r>1$ with $\frac{1}{t}+\frac{\mu}{N}+\frac{1}{r}=2$, $f\in L^{t}(\mathbb{R}^{N})$ and $h\in L^{r}(\mathbb{R}^{N})$. Then there exists a constant ${C}(N,\mu,t,r)>0$ independent of $f$ and $h$ such that
\begin{equation}\label{HLS}
 \displaystyle\int_{\mathbb{R}^N}\int_{\mathbb{R}^N}\frac{f(x)h(y)}{|x-y|^\mu}dxdy\leq C(N,\mu,t,r)\|f\|_{L^t(\mathbb{R}^N)}\|h\|_{L^r(\mathbb{R}^N)}.
\end{equation}
If $t=r=\frac{2N}{2N-\mu}$, then ${C}(N,\mu,t,r)=C(N,\mu)=\frac{\pi^{\frac{\mu}{2}}\Gamma(\frac{N-\mu}{2})}{\Gamma(
 \frac{2N-\mu}{2})}\Big(\frac{\Gamma(N)}{\Gamma(\frac{N}{2})}\Big)^{\frac{N-\mu}{N}}$,
and there is equality in \eqref{HLS} if and only if $f\equiv (\mathrm{const}.)h$ and $h(x)=c\big(\frac{\lambda}{1+\lambda^{2}|x-a|^{2}}\big)
^{\frac{2N-\mu}{2}}$ for some $c\in \mathbb{R}$, $\lambda>0$, $a\in\mathbb{R}^{N}$,
where
$\Gamma$ denotes the Gamma function $\Gamma(\gamma)=\int_0^{+\infty}t^{\gamma-1}e^{-t}dt$ for $\gamma>0$.

{This sharp inequality naturally leads to a quantitative remainder problem. Through a duality principle based on quantitative convexity, Carlen \cite{C17} transferred stability from fractional Sobolev inequalities to the HLS inequality. In particular,} for all $N\geq3$ and $0<\mu<N$, there exists a constant $k_{BE}>0$ such that
\begin{equation*}
C(N,\mu)\|u\|^2_{L^{\frac{2N}{2N-\mu}}(\mathbb{R}^N)}-
\int_{\mathbb{R}^{N}}\int_{\mathbb{R}^{N}}
\frac{u(x)u(y)}{|x-y|^\mu}dxdy
\geq k_{BE} \inf\limits_{v\in \mathcal{M}_\mu} \|u-v\|^2_{L^{\frac{2N}{2N-\mu}}(\mathbb{R}^N)},
 \quad \forall \, u\in L^{\frac{2N}{2N-\mu}}(\mathbb{R}^{N}),
\end{equation*}
where $\mathcal{M}_\mu=\Big\{c\big(\frac{\lambda}{1+\lambda^{2}|x-a|^{2}}\big)
^{\frac{2N-\mu}{2}}:c\in \mathbb{R},\lambda>0,a\in\mathbb{R}^{N}\Big\}$
is the manifold of all extremal functions for the HLS inequality. In \cite{CLT24'}, Chen et al. explored the explicit lower bound of $k_{BE}$. Furthermore, in \cite{CLT25,CLT24}, they established the
optimal asymptotic lower bound of $k_{BE}$ as the dimension $N\rightarrow+\infty$.

{Combining \eqref{HLS} with \eqref{Sobolev inequality} shows that the nonlocal energy} $\int_{\mathbb{R}^{N}}\int_{\mathbb{R}^{N}} \frac{|u(x)|^p|u(y)|^p}{|x-y|^\mu}dxdy$
is well defined in $D^{1,2}(\mathbb{R}^{N})$ if
$p\in\big[\frac{2N-\mu}{N},\frac{2N-\mu}{N-2}\big]$.
{Accordingly, $2^{\ast}_{\mu}=(2N-\mu)/(N-2)$ is called the upper critical exponent. At this endpoint, one obtains the sharp nonlocal Sobolev inequality}
\begin{equation}\label{NonS}
 \mathcal{S}_{HL}(N,\mu) \left(
\int_{\mathbb{R}^{N}}\int_{\mathbb{R}^{N}}
\frac{|u(x)|^{2^{\ast}_{\mu}}|u(y)|^{2^{\ast}_{\mu}}}{|x-y|^\mu}dxdy
\right)^{\frac{1}{2^{\ast}_{\mu}}}\leq \int_{\mathbb{R}^{N}}|\nabla u|^{2}d x,\quad \forall \, u\in D^{1,2}(\mathbb{R}^{N}),
\end{equation}
where the best constant is given by $\mathcal{S}_{HL}(N,\mu)=\mathcal{S}(N)C(N,\mu)^{-\frac{1}{2^{\ast}_{\mu}}}$.
The associated Euler--Lagrange equation becomes
\begin{equation}\label{limit2'}
 -\Delta u=\left(\int_{\mathbb{R}^{N}}\frac{|u(y)|
^{2^{\ast}_{\mu}}}{|x-y|^{\mu}}dy \right)|u|^{2^{\ast}_{\mu}-2}u,~~u\in D^{1,2}(\mathbb{R}^{N}).
\end{equation}
{For \eqref{NonS}, Deng et al. \cite{DTYZ23} proved a Bianchi--Egnell-type estimate, using the classification of positive solutions \cite{DY19,GHPS19} and the nondegeneracy of the bubble \cite{LLTX23} for \eqref{limit2'}. More precisely,}
if $N\geq3$, $0<\mu<N$ with
 $\mu\leq 4$, there are constants $c_{UB}>c_{NS}>0$ such that
 for every $u\in{D}^{1,2}(\mathbb{R}^{N})$,
\begin{equation*}
c_{UB} ~ \mathrm{dist}(u,\mathcal{M})^2\geq \int_{\mathbb{R}^{N}}|\nabla u|^{2}{d}x-\mathcal{S}_{HL}(N,\mu)\left(
\int_{\mathbb{R}^{N}}\int_{\mathbb{R}^{N}}
\frac{|u(x)|^{2^{\ast}_{\mu}}|u(y)|^{2^{\ast}_{\mu}}}{|x-y|^\mu}dxdy
\right)^{\frac{1}{2^{\ast}_{\mu}}}\geq c_{NS} ~ \mathrm{dist}(u,\mathcal{M})^2.
\end{equation*}
Subsequent studies have extended this result to the second Sobolev inequality \cite{YZ25} and the fractional Sobolev framework \cite{DLYZ25,LYZ25}.
More recently, Zhang \cite{Z26} proved strict spectral and two-bubble upper bounds for the Euclidean nonlocal quotient and established existence of an extremizer for the optimal stability constant. See also \cite{DT25}.

{We now turn to the Heisenberg group $\mathbb{H}^{n}$, whose notation is summarized in Section \ref{secnp}. Folland and Stein \cite{FS74} established the sharp Sobolev inequality}
\begin{equation}\label{folland-stein}
 \int_{\mathbb{H}^{n}}|\nabla_{\mathbb{H}}u|^{2}{d}\xi\geq S(Q)\left(\int_{\mathbb{H}^{n}}|u|^{Q^{\ast}}{d}\xi\right)^{\frac{2}{Q^{\ast}}},\quad \forall \, u\in S^{1,2}(\mathbb{H}^{n}).
\end{equation}
Moreover, it follows from Jerison and Lee \cite{JL88} that \[S(Q)=\frac{4\pi n^2}{(n!4^n)^{\frac{1}{n+1}}},\]
{and the extremals are precisely given by the following family of Jerison--Lee bubbles:}
\begin{equation}\label{gU}
 \mathfrak{g}_{\lambda,\xi_0}U(\xi)=\lambda^{\frac{Q-2}{2}}U(\delta_\lambda(\tau_{{\xi_0}^{-1}}(\xi))),
 ~~\lambda>0,~\xi_{0}\in\mathbb{H}^{n},
\end{equation}
{where the group actions are defined in Section \ref{secnp} and} \[U(\xi)=U(z,t)=\frac{(2n)^n}{[(1+|z|^{2})^{2}+t^{2}]^{\frac{Q-2}{4}}}\]
is the unique (up to translations and scalings, i.e., $\mathfrak{g}_{\lambda,\xi}U$ for all $\lambda>0$ and $\xi\in\mathbb{H}^{n}$)
positive solution of the Euler--Lagrange equation
\begin{equation*}
 -\Delta_{\mathbb{H}} u=|u|^{Q^{\ast}-2}u,~~\xi\in\mathbb{H}^{n}.
\end{equation*}
{Using the correspondence between spectral problems on $\mathbb{H}^{n}$ and on the CR sphere \cite{MU02}, Loiudice \cite{L05} established quantitative stability for \eqref{folland-stein}: there exists $H_{BE}>0$ such that}
\begin{equation*}
 \int_{\mathbb{H}^{n}}|\nabla_{\mathbb{H}}u|^{2}{d}\xi- S(Q)\left(\int_{\mathbb{H}^{n}}|u|^{Q^{\ast}}{d}\xi\right)^{\frac{2}{Q^{\ast}}}\geq H_{BE} ~ \mathrm{dist}(u,\mathfrak{M})^2,\quad \forall \, u\in S^{1,2}(\mathbb{H}^{n}),
\end{equation*}
 where $\mathfrak{M}=\big\{c\mathfrak{g}_{\lambda,\xi}U: c\in \mathbb{R},\lambda>0,\xi\in \mathbb{H}^n\big\}$ denotes the manifold of Jerison--Lee bubbles, and the distance is defined as $\mathrm{dist}(u,\mathfrak{M})=\inf_{c\in \mathbb{R},\lambda>0,\xi\in \mathbb{H}^n}\|u-c\mathfrak{g}_{\lambda,\xi}U\|_{S^{1,2}(\mathbb{H}^{n})}$.
 A fractional counterpart can be found in \cite{LZ15}. {For the optimal constant $H_{BE}$, Tang, Zhang and Zhang \cite{TZZ24} proved strict upper bounds at both possible compactness thresholds. The spectral threshold is}
\begin{equation*}
 H_{BE}^{\mathrm{spec}}=\frac{2}{n+4},
\end{equation*}
{and the two-bubble threshold is}
\begin{equation*}
 H_{BE}^{2\text{-peak}}=2-2^{\frac{n}{n+1}}.
\end{equation*}
{These strict inequalities provide the compactness input for attainment of $H_{BE}$. An explicit lower bound was later obtained through a CR Yamabe-flow approach in \cite{CLTW25}.}

\begingroup
The preceding literature isolates the gap addressed in this paper. The Euclidean nonlocal problem has analogous spectral and dichotomy thresholds together with an attainment theory \cite{DT25,Z26}, while the local Heisenberg problem has a related variational structure \cite{TZZ24}. Li \cite{Li26} recently showed that a fourth-order normal perturbation can improve a local spectral threshold when the cubic contribution vanishes. In the present nonlocal Heisenberg setting, however, the HLS double integral introduces additional mode-dependent multipliers, so the fourth-order coefficients and the optimal correction must be recomputed.

Three points require particular care. First, after the Cayley transform, the linearization at a bubble is governed by two coupled eigenvalue sequences on the complex spherical harmonics $\mathcal H_{i,j}^{n+1}$. Second, the cubic moment of the lowest real normal mode vanishes by the $U(1)$-symmetry; consequently, strictness at the local spectral threshold requires a fourth-order expansion with an optimized second-order normal correction. Third, the compactness argument requires a precise expansion of the interaction between translated and dilated bubbles through the kernel $|\eta^{-1}\xi|^{-\mu}$.

Accordingly, the paper has four objectives: to identify the local spectral and two-bubble compactness thresholds; to prove attainment of $H_{NS}$ in the parameter ranges stated below; to establish the strict comparison $H_{NS}>H_{BE}$ and the sharp universal upper defect constant $1$; and to study the corresponding critical-point residual quotient. The precise statements and the main technical ingredients are given in Section~\ref{sec main}.
\endgroup

\section{Notation and preliminaries}\label{secnp}

The Heisenberg group $\mathbb{H}^{n}$ is $\mathbb{C}^{n}\times\mathbb{R}$ with elements $\xi=(\xi_{l})=(z,t)$, $\xi'=(\xi'_{l})=(z',t')$, $1\leq l\leq 2n+1$, and group operation \[\xi'\xi=\big(z+z',t+t'+2\mathrm{Im}z'\cdot\bar{z}\big).\]
The left translations are given by $\tau_{\xi'}(\xi)=\xi'\xi$,
{and the group dilations are}
$\{\delta_{\lambda}\}_{\lambda>0}: \mathbb{H}^{n}\rightarrow\mathbb{H}^{n}$, $\delta_{\lambda}(\xi)=(\lambda z,\lambda^{2}t)$.
Define the homogeneous norm $|\xi|=(|z|^{4}+t^{2})^{\frac{1}{4}}$, and the distance
$d(\xi,\xi')=|(\xi')^{-1}\xi|$.
It holds that $|\delta_{\lambda}(\xi)|=\lambda|\xi|$ and
{$d(\delta_\lambda\xi,\delta_\lambda\xi')=\lambda d(\xi,\xi')$.}
As usual, the homogeneous dimension of $\mathbb{H}^{n}$ is $Q=2n+2$. Denote by $B(\xi_0,r)$ the ball of radius $r$ centered at $\xi_0$ with respect to the Heisenberg distance $d$.
The canonical left-invariant vector fields on $\mathbb{H}^{n}$ are
\begin{equation*}
 X_{j}=\frac{\partial }{\partial x_{j}}+2y_{j}\frac{\partial }{\partial t},\quad X_{n+j}=\frac{\partial}{\partial y_{j}}-2x_{j}\frac{\partial}{\partial t},\quad j=1, \ldots, n.
\end{equation*}
The horizontal gradient is defined by
\begin{equation*}
 \nabla_{\mathbb{H}}=(X_{1}, \ldots, X_{n},X_{n+1}, \ldots, X_{2n}),
\end{equation*}
and the Kohn Laplacian (or sub-Laplacian) operator is
\begin{equation*}
 \Delta_{\mathbb{H}}=\mathop{\sum}\limits_{j=1}^{n}\big(X_{j}^{2}+X_{n+j}^{2}\big).
\end{equation*}
Let $Q^{\ast}=\frac{2Q}{Q-2}$, define the standard Folland--Stein--Sobolev space $S^{1,2}(\mathbb{H}^{n})=\big\{u\in L^{Q^{\ast}}(\mathbb{H}^{n}):\nabla_{\mathbb{H}}u\in L^{2}(\mathbb{H}^{n})\big\}$,
with the inner product $\langle u,v\rangle_{S^{1,2}(\mathbb{H}^{n})}=\int_{\mathbb{H}^{n}}\nabla_{\mathbb{H}}u\cdot\nabla_{\mathbb{H}}v{d}\xi$,
and the corresponding norm $\|u\|_{S^{1,2}(\mathbb{H}^{n})}=\big(\int_{\mathbb{H}^{n}}|\nabla_{\mathbb{H}}u|^{2}{d}\xi\big)^{\frac{1}{2}}$.
Throughout the paper, all function spaces are understood over the real numbers unless explicitly stated otherwise.

{The Heisenberg-group analogue of the HLS inequality was established} by Folland and Stein \cite{FS74} as well as Frank and Lieb \cite{FL12}.
\begin{proposition}\label{pro:HLS}
 Suppose that $Q\geq 4$, $0<\mu< Q$ and $t,r>1$ with $\frac{1}{t}+\frac{\mu}{Q}+\frac{1}{r}=2$, $f\in L^{t}(\mathbb{H}^{n})$ and $h\in L^{r}(\mathbb{H}^{n})$. Then there exists a constant $\widetilde{C}(Q,\mu,t,r)>0$ independent of $f$ and $h$ such that
 \begin{equation}\label{eq:HLSH} \int_{\mathbb{H}^{n}}\int_{\mathbb{H}^{n}}\frac{f(\xi)h(\eta)}{|\eta^{-1}\xi|^{\mu}}
d\xi{d}\eta\leq \widetilde{C}(Q,\mu,t,r)\|f\|_{L^t(\mathbb{H}^{n})}\|h\|_{L^r(\mathbb{H}^{n})}.
 \end{equation}
 If $t=r=\frac{2Q}{2Q-\mu}$, then
\begin{equation*}
 \widetilde{C}(Q,\mu,t,r)=C(Q,\mu)=\bigg(\frac{\pi^{n+1}}{2^{n-1}n!}\bigg)^{\frac{\mu}{Q}}
 \frac{n!\Gamma(\frac{Q-\mu}{2})}{\Gamma^2(\frac{2Q-\mu}{4})},
\end{equation*}
and there is equality in \eqref{eq:HLSH} if and only if $f\equiv (const.)h$ and $h(\xi)=c\lambda^{\frac{2Q-\mu}{2}}V(\delta_\lambda(\tau_{{\xi_0}^{-1}}(\xi)))$ for some $c\in \mathbb{R}$, $\lambda>0$, $\xi_{0}\in\mathbb{H}^{n}$, and $V(\xi)=\frac{1}{[(1+|z|^{2})^{2}+t^{2}]^{\frac{2Q-\mu}{4}}}$.
\end{proposition}

Similarly, by \eqref{eq:HLSH}, $Q^{\ast}_{\mu}=\frac{2Q-\mu}{Q-2}$ is called the upper critical exponent on the Heisenberg group.
{At the upper critical exponent, the sharp inequality, the classification of extremals, and the associated Euler--Lagrange equation take the following form \cite{YZ251}; see also \cite{ZWZLX25}.}
\begin{lemma}
 Let $Q\geq 4$, $0<\mu<Q$. Then for any $u\in S^{1,2}(\mathbb{H}^{n})\setminus\{0\}$, the inequality
 \begin{equation}\label{eq:HLS'}
 S_{HL}(Q,\mu) \left(\int_{\mathbb{H}^{n}}\int_{\mathbb{H}^{n}}\frac{|u(\xi)|^{Q^{\ast}_{\mu}}|u(\eta)|
^{Q^{\ast}_{\mu}}}{|\eta^{-1}\xi|^{\mu}}{d}\xi{d}\eta\right)^{\frac{1}{Q^{\ast}_{\mu}}}\leq \|\nabla_{\mathbb{H}}u\|^2_{L^2(\mathbb{H}^{n})}
 \end{equation}
holds with the sharp constant
\begin{equation}\label{relation}
 S_{HL}(Q,\mu)=S(Q)C(Q,\mu)^{-\frac{1}{Q^{\ast}_{\mu}}}.
\end{equation}
Equality in \eqref{eq:HLS'} holds
if and only if \[u(\xi)=c\mathfrak{g}_{\lambda,\xi_0}U(\xi)\] for some $c\in\mathbb{R}\setminus\{0\}$, $\lambda>0$ and $\xi_0\in \mathbb{H}^n$, where $\mathfrak{g}_{\lambda,\xi_0}U(\xi)$ is defined by \eqref{gU}. Moreover,
 \begin{equation*}
 U_\mu(\xi)=\mathcal{A}U(\xi)\quad \mathrm{with}\quad \mathcal{A}=S(Q)^{\frac{(Q-\mu)(2-Q)}{4(Q+2-\mu)}}C(Q,\mu)^{\frac{2-Q}{2(Q+2-\mu)}}
\end{equation*}
is the unique (up to translations and scalings, i.e., $\mathfrak{g}_{\lambda,\xi}U_\mu$ for all $\lambda>0$ and $\xi\in\mathbb{H}^{n}$)
 positive solution of the Euler--Lagrange equation
\begin{equation}\label{limit3}
 -\Delta_{\mathbb{H}} u=\left(\int_{\mathbb{H}^{n}}\frac{|u(\eta)|
^{Q^{\ast}_{\mu}}}{|\eta^{-1}\xi|^{\mu}}{d}\eta\right)|u|^{Q^{\ast}_{\mu}-2}u,~~\xi\in\mathbb{H}^{n}.
\end{equation}
\end{lemma}
For equation \eqref{limit3}, Yang and Zhang \cite{YZ251} established
the following nondegeneracy property of bubbles.
\begin{lemma}
 Let $Q\geq 4$, $0<\mu< Q$.
 {If $v \in S^{1,2}(\mathbb{H}^{n})$ solves the linearized equation}
 \begin{equation*}
 -\Delta_{\mathbb{H}} v=Q^{\ast}_{\mu}\left(\int_{\mathbb{H}^{n}}\frac{|U(\eta)|^{Q^{\ast}_{\mu}-1}v(\eta)}
{|\eta^{-1}\xi|^{\mu}}{d}\eta\right)|U|^{Q^{\ast}_{\mu}-2}U+(Q^{\ast}_{\mu}-1)
\left(\int_{\mathbb{H}^{n}}\frac{|U(\eta)|^{Q^{\ast}_{\mu}}}{|\eta^{-1}\xi|^{\mu}}{d}\eta\right)
|U|^{Q^{\ast}_{\mu}-2}v,
\end{equation*}
{then $v$ is a linear combination} of the functions $\{Z^{a}\}_{a=1}^{2n +2}$, where $Z^a$ are defined by
 \begin{equation*}
 Z^{a}=\frac{\partial \mathfrak{g}_{1, \eta}U }{\partial\eta^{(a)}}\bigg|_{ \eta=0},~~ a=1,\ldots, 2n+1,
 \end{equation*}
\begin{equation*}
\begin{aligned}
Z^{2n+2}= \frac{\partial \mathfrak{g}_{r, 0}U }{\partial r}\bigg|_{r=1}=
\frac{Q-2}{2}U-(Q-2)\frac{|z|^2(1+|z|^2)+t^2}{(1+|z|^2)^2+t^2} U,
\end{aligned}
\end{equation*}
and $\eta^{(a)}$ denotes the $a$-th coordinate of $\eta \in \mathbb{H}^n$.
\end{lemma}

We now collect several preliminary results, which are largely adapted from \cite[Section 5]{FL12} and \cite[Section 2]{YZ251}.
We regard the sphere $\mathbb{S}^{2n+1}$ as a submanifold of $\mathbb{C}^{n+1}$ equipped with coordinates
$(\zeta_1,\ldots,\zeta_{n+1})$ subject to the constraint $\sum\limits_{j=1}^{n+1}|\zeta_j|^2=1$.
 The Cayley transform $\mathcal{C}:\mathbb{H}^n \rightarrow \mathbb{S}^{2n+1}\setminus \{(0,0,\ldots,0,-1)\}$ is defined by
\begin{equation*}
 \mathcal{C}(z,t)=\left(\frac{2z}{1+|z|^2+it},\frac{1-|z|^2-it}{1+|z|^2+it}\right),
\end{equation*}
and its inverse map $\mathcal{C}^{-1}: \mathbb{S}^{2n+1}\setminus \{(0,0,\ldots,0,-1)\}\rightarrow \mathbb{H}^n$ is given by
\begin{equation*}
 \mathcal{C}^{-1}(\zeta)=\left(\frac{\zeta_1}{1+\zeta_{n+1}},\ldots,\frac{\zeta_n}{1+\zeta_{n+1}},\mathrm{Im}\frac{1-\zeta_{n+1}}{1+\zeta_{n+1}}\right).
\end{equation*}
{Its Jacobian determinant is}
\begin{equation*}
 \mathcal{J}_\mathcal{C}(z,t)=\frac{2^{2n+1}}{[(1+|z|^2)^2+t^2]^{n+1}}=2^{Q-1}(2n)^{-2n-2}U^{Q^{\ast}},
\end{equation*}
which yields the change-of-variable formula \[\int_{\mathbb{S}^{2n+1}}F(\zeta)d \sigma(\zeta)=\int_{\mathbb{H}^n}F(\mathcal{C}\xi) \mathcal{J}_\mathcal{C}(\xi) d \xi\]
valid for all $F\in L^1(\mathbb{S}^{2n+1})$.

For $\zeta=\mathcal{C}(\xi)=\mathcal{C}(z,t)$ and $\zeta'=\mathcal{C}(\eta)=\mathcal{C}(z',t')$, the identity
\begin{equation*}
 |1-\zeta\cdot\overline{\zeta'}|=2\big[(1+|z|^2)^2+t^2\big]^{-\frac{1}{2}}|\xi^{-1}\eta|^2\big[(1+|z'|^2)^2+(t')^2\big]^{-\frac{1}{2}}
\end{equation*}
holds in view of \cite[Appendix A]{FL12}.
 Given any function $f:\mathbb{H}^n\rightarrow \mathbb{R}$, we define its pushforward $\mathcal{C}_*f:\mathbb{S}^{2n+1}\setminus \{(0,0,\ldots,0,-1)\}\rightarrow \mathbb{R}$ via
\begin{equation*}
 \mathcal{C}_*f(\zeta)=[\mathcal{J}_\mathcal{C}(\xi)]^{-\frac{1}{Q^{\ast}}}f(\mathcal{C}^{-1}\zeta).
\end{equation*}
Likewise, for every $F:\mathbb{S}^{2n+1}\setminus \{(0,0,\ldots,0,-1)\}\rightarrow \mathbb{R}$, the pullback
$\mathcal{C}^*F:\mathbb{H}^n\rightarrow \mathbb{R}$ is defined by
\begin{equation*}
 \mathcal{C}^*F(\xi)=[\mathcal{J}_\mathcal{C}(\xi)]^{\frac{1}{Q^{\ast}}}F(\mathcal{C}\xi).
\end{equation*}

{When the complexified spectral decomposition is used below, both $\mathcal{C}_*$ and $\mathcal{C}^*$ are understood by complex-linear extension.}

{We next recall the complex spherical harmonics used to diagonalize the integral operators on}
$\mathbb{S}^{2n+1}$; see \cite[Section 5.2]{FL12}. {The complexified space $L^2(\mathbb{S}^{2n+1};\mathbb C)$
 admits the orthogonal direct-sum decomposition}
\begin{equation*}
 {L^2(\mathbb{S}^{2n+1};\mathbb C)=\bigoplus\limits_{i,j\geq0}\mathcal{H}_{i,j}^{n+1}.}
\end{equation*}
where $\mathcal{H}_{i,j}^{n+1}$ stands for the space formed by restricting to $\mathbb{S}^{2n+1}$ all harmonic polynomials $p(z,\bar{z})$ on $\mathbb{C}^{n+1}$ that are
homogeneous of degree $i$ in $z$ and homogeneous of degree $j$ in $\bar{z}$, and
\begin{equation*}
 \mathrm{dim} \mathcal{H}_{i,j}^{n+1}=\frac{(i+j+n)(i+n-1)!(j+n-1)!}{i!j!n!(n-1)!}.
\end{equation*}
For later use on the real function space, we set
\[
 (\mathcal H_{i,i}^{n+1})_{\mathbb R}
 :=\mathcal H_{i,i}^{n+1}\cap L^2(\mathbb S^{2n+1};\mathbb R),
 \qquad
 (\mathcal H_{i,j}^{n+1}\oplus\mathcal H_{j,i}^{n+1})_{\mathbb R}
 :=(\mathcal H_{i,j}^{n+1}\oplus\mathcal H_{j,i}^{n+1})
 \cap L^2(\mathbb S^{2n+1};\mathbb R)
\]
for $i\ne j$.
For each $\mu\in (0,Q)$, and integers $i,j\geq0$, we define
\begin{equation*}
 E_{i,j}(\mu)=E_{j,i}(\mu)=\frac{2\pi^{n+1}\Gamma(n+1-\frac{\mu}{2})}{\Gamma^2(\frac{\mu}{4})}
 \frac{\Gamma(i+\frac{\mu}{4})\Gamma(j+\frac{\mu}{4})}{\Gamma(i+n+1-\frac{\mu}{4})\Gamma(j+n+1-\frac{\mu}{4})}.
\end{equation*}
In particular,
\begin{equation*}
 E_{0,0}(\mu)=2\pi^{n+1}\frac{\Gamma(n+1-\frac{\mu}{2})}{\Gamma^2(n+1-\frac{\mu}{4})},\quad E_{1,0}(\mu)=\frac{\mu}{4n+4-\mu}E_{0,0}(\mu),
\end{equation*}
\begin{equation*}
 E_{2,0}(\mu)=\frac{\mu(\mu+4)}{(4n+4-\mu)(4n+8-\mu)}E_{0,0}(\mu)>
 E_{1,1}(\mu)=\frac{\mu^2}{(4n+4-\mu)^2}E_{0,0}(\mu).
\end{equation*}
Making use of the Funk-Hecke formula; see \cite[Corollary 5.3]{FL12},
\begin{equation*}
 \int_{\mathbb{S}^{2n+1}}\frac{1}{|1-\zeta\cdot \overline{\zeta'}|^{\frac{\mu}{2}}}Y(\zeta')d \sigma(\zeta')=E_{i,j}(\mu) Y(\zeta),\qquad \mathrm{for~all}~Y \in \mathcal{H}_{i,j}^{n+1},
\end{equation*}
Yang and Zhang proved in \cite[Lemma 2.3]{YZ251} that for every $Y \in \mathcal{H}_{i,j}^{n+1}$,
\begin{equation}\label{FHA}
 \int_{\mathbb{S}^{2n+1}}\int_{\mathbb{S}^{2n+1}}\frac{1}{|1-\zeta\cdot \overline{\zeta'}|^{\frac{Q-2}{2}}}\frac{1}{|1-\zeta'\cdot \overline{\zeta''}|^{\frac{\mu}{2}}}Y(\zeta'')d \sigma( \zeta'')d \sigma(\zeta')=E_{i,j}(Q-2)E_{i,j}(\mu) Y(\zeta),
\end{equation}
\begin{equation*}
 \int_{\mathbb{S}^{2n+1}}\int_{\mathbb{S}^{2n+1}}\frac{1}{|1-\zeta\cdot \overline{\zeta'}|^{\frac{Q-2}{2}}}\frac{1}{|1-\zeta'\cdot \overline{\zeta''}|^{\frac{\mu}{2}}}Y(\zeta')d \sigma(\zeta'')d \sigma(\zeta')=E_{i,j}(Q-2)E_{0,0}(\mu) Y(\zeta).
\end{equation*}

\section{Main results}\label{sec main}

{We now state the main results. The functional part identifies the two compactness thresholds, proves attainment, and determines the sharp reverse comparison. The final theorem concerns the residual quotient for the Euler--Lagrange equation and is logically independent of attainment of the functional quotient.}

{\bf{In the functional setting.}}
Building upon the remainder estimate of Deng et al. \cite{DTYZ23}, we recently established in \cite[Theorem 1.1]{CW26} the quantitative stability of \eqref{eq:HLS'} for $0<\mu<Q$ with $\mu\leq4$. More precisely, there exist constants $H_{UB}>H_{NS}>0$ such that for every $u\in S^{1,2}(\mathbb{H}^{n})$,
\begin{multline}\label{e1.15}
 H_{UB}~ \mathrm{dist}(u,\mathfrak{M})^2\geq \int_{\mathbb{H}^{n}}|\nabla_{\mathbb{H}} u|^{2}{d}\xi-S_{HL}(Q,\mu) \left(\int_{\mathbb{H}^{n}}\int_{\mathbb{H}^{n}}\frac{|u(\xi)|^{Q^{\ast}_{\mu}}|u(\eta)|
^{Q^{\ast}_{\mu}}}{|\eta^{-1}\xi|^{\mu}}{d}\xi{d}\eta\right)^{\frac{1}{Q^{\ast}_{\mu}}}\\
\geq H_{NS} ~ \mathrm{dist}(u,\mathfrak{M})^2.
\end{multline}
{Combining the sharp HLS inequality with the local stability estimate of Tang, Zhang and Zhang \cite{TZZ24} extends the lower estimate in \eqref{e1.15} to every $0<\mu<Q$ and gives the non-strict comparison $H_{NS}\geq H_{BE}$. The purpose of the present paper is to identify the strict thresholds that restore compactness, prove attainment in explicit parameter ranges, and upgrade this comparison to $H_{NS}>H_{BE}$.}
Define
\begin{equation}\label{defH}
 H_{NS}:=\inf\limits_{u\in S^{1,2}(\mathbb{H}^{n})\setminus \mathfrak{M}}\mathcal{L}(u),
\end{equation}
where
\begin{equation*}
 \mathcal{L}(u)=\frac{\int_{\mathbb{H}^{n}}|\nabla_{\mathbb{H}} u|^{2}{d}\xi-S_{HL}(Q,\mu) \Big(\int_{\mathbb{H}^{n}}\int_{\mathbb{H}^{n}}\frac{|u(\xi)|^{Q^{\ast}_{\mu}}|u(\eta)|
^{Q^{\ast}_{\mu}}}{|\eta^{-1}\xi|^{\mu}}{d}\xi{d}\eta\Big)^{\frac{1}{Q^{\ast}_{\mu}}}}{\mathrm{dist}(u,\mathfrak{M})^2}.
\end{equation*}

We first identify the two strict thresholds that control the possible loss of compactness. The first is generated by a local perturbation of one bubble and the second by two asymptotically separated bubbles.
\begin{theorem}\label{main thm0}
\begingroup
Let $Q=2n+2\ge4$ and $0<\mu<Q$. Then
\begin{equation}\label{upper2}
H_{NS}<H_{NS}^{2\text{-peak}}
:=2-2^{\frac{2n}{4n+4-\mu}}.
\end{equation}
Moreover, let $\Gamma_{n,\mu}$ be the explicit fourth-order coefficient defined in \eqref{Gamma-def}. If
\begin{equation}\label{Gamma-condition}
\Gamma_{n,\mu}<0,
\end{equation}
then
\begin{equation}\label{upper1}
H_{NS}<H_{NS}^{\mathrm{spec}}
:=\frac{(4n+8-2\mu)(\mu+4)}{2(n+4)(4n+8-\mu)}.
\end{equation}
In particular, $\Gamma_{n,2}<0$ for every $n\ge1$; see Lemma~\ref{Gamma-mu-two}.
\endgroup
\end{theorem}

Whenever \eqref{Gamma-condition} holds, the two strict inequalities in Theorem~\ref{main thm0} provide the compactness barriers used below: the spectral bound rules out minimizing sequences approaching the bubble manifold, whereas the two-bubble bound rules out dichotomy.

\begin{theorem}\label{main thm1}
\begingroup
Let \(Q\ge4\) and \(0<\mu<Q\). Assume either
\[
\mu\le4\quad\text{and}\quad \Gamma_{n,\mu}<0,
\]
or
\begin{equation}\label{ach}
\mu>4,\quad n\ge5,\quad
\mu \ge 4n+4-\frac{2n}{\log_2\frac{2n+6}{n+4}}.
\end{equation}
Then \(H_{NS}\) is attained: there exists
\(u_0\in S^{1,2}(\mathbb{H}^n)\setminus\mathfrak M\) such that
\[
\mathcal L(u_0)=H_{NS}.
\]
\endgroup
\end{theorem}

The preceding theorems concern the smallest constant in the lower stability estimate. For completeness, we also determine the opposite, universal upper comparison between the deficit and the squared distance.
\begin{theorem}\label{main thm2}
Let $Q\geq 4$ and $0<\mu<Q$. {Then every $u\in S^{1,2}(\mathbb{H}^{n})$ satisfies}
\begin{equation}\label{ub1}
 \int_{\mathbb{H}^{n}}|\nabla_{\mathbb{H}} u|^{2}{d}\xi-S_{HL}(Q,\mu) \left(\int_{\mathbb{H}^{n}}\int_{\mathbb{H}^{n}}\frac{|u(\xi)|^{Q^{\ast}_{\mu}}|u(\eta)|
^{Q^{\ast}_{\mu}}}{|\eta^{-1}\xi|^{\mu}}{d}\xi{d}\eta\right)^{\frac{1}{Q^{\ast}_{\mu}}}\leq 1\cdot \mathrm{dist}(u,\mathfrak{M})^2.
\end{equation}
Furthermore, the constant $1$ is sharp, and equality holds if and only if $u\in \mathfrak{M}$.
\end{theorem}

{The following remarks clarify the parameter ranges and the relation with the local problem.}
\begin{remark}\label{rmk2}
{\rm
{The strict spectral upper bound in Theorem~\ref{main thm0} is conditional only on the explicit sign $\Gamma_{n,\mu}<0$. Lemma~\ref{Gamma-mu-two} provides the concrete example $\Gamma_{n,2}<0$ for every $n\ge1$. Since the coefficients in \eqref{Gamma-def} depend continuously on $\mu$ and $d_\alpha>0$, for each fixed $n$ the same sign persists on an open interval containing $\mu=2$. Intersecting this interval with $(0,4]$ gives a nonempty parameter range to which the first alternative in Theorem~\ref{main thm1} applies.}

$\bullet$ In Theorem~\ref{main thm1}, using Lemma~\ref{num}, we obtain
\begin{equation*}
 4<4n+4-\frac{2n}{\log_2 \frac{2n+6}{n+4}}<Q=2n+2\Longleftrightarrow n\geq5.
\end{equation*}
Using \eqref{eq:HLSH} and \eqref{relation}, we arrive at
\begin{equation*}
 S_{HL}(Q,\mu) \left(\int_{\mathbb{H}^{n}}\int_{\mathbb{H}^{n}}\frac{|u(\xi)|^{Q^{\ast}_{\mu}}|u(\eta)|
^{Q^{\ast}_{\mu}}}{|\eta^{-1}\xi|^{\mu}}{d}\xi{d}\eta\right)^{\frac{1}{Q^{\ast}_{\mu}}}\leq
S_{HL}(Q,\mu) C(Q,\mu)^{\frac{1}{Q^{\ast}_{\mu}}} \|u\|_{L^{Q^{\ast}}(\mathbb{H}^n)}^2=S(Q)\|u\|_{L^{Q^{\ast}}(\mathbb{H}^n)}^2.
\end{equation*}
Thus, relying on Theorem~\ref{main thm1}, Proposition \ref{pro:HLS} and $u_0\in S^{1,2}(\mathbb{H}^{n})\setminus \mathfrak{M}$, we further derive the strict inequality
\begin{equation}\label{dl}
 H_{NS}>H_{BE}.
\end{equation}

$\bullet$ In \cite{LZ15}, Liu and Zhang derived a sharp universal upper bound associated with the classical Folland--Stein--Sobolev inequality. More precisely, they proved that
for every $u\in S^{1,2}(\mathbb{H}^{n})$,
\begin{equation}\label{ub2}
 \int_{\mathbb{H}^{n}}|\nabla_{\mathbb{H}} u|^{2}{d}\xi-S(Q)\left(\int_{\mathbb{H}^{n}}|u|^{Q^{\ast}}d\xi\right)^{\frac{2}{Q^{\ast}}}\leq 1\cdot \mathrm{dist}(u,\mathfrak{M})^2.
\end{equation}
Moreover, the constant $1$ on the right-hand side is optimal and equality holds in \eqref{ub2} if and only if $u\in \mathfrak{M}$. Comparing estimates \eqref{ub1} and \eqref{ub2}, we observe that the classical Folland--Stein--Sobolev inequality and its nonlocal counterpart admit identical upper bounds, while their lower bounds differ as demonstrated in \eqref{dl}.
}
\end{remark}

{\bf In the critical-point setting.} In \cite[Theorems 1.9-1.10]{ZXW25} and \cite[Corollary 1.5]{CW26},
quantitative stability results for critical points of equation \eqref{limit3} were established. More precisely, the corresponding conclusions read as follows:
Let either $Q\geq4$, $\mu\in(0,Q)$ with $m=1$ or $Q=4$, $\mu\in(0,4)$ with $m\geq2$.
For any nonnegative function $u\in S^{1,2}(\mathbb{H}^{n})$ satisfying
\begin{equation}\label{addbound}
 \Big(m-\frac{1}{2}\Big)S_{HL}(Q,\mu)^{\frac{2Q-\mu}{Q+2-\mu}}\leq \int_{\mathbb{H}^{n}}|\nabla_{\mathbb{H}} u|^2d\xi
 \leq \Big(m+\frac{1}{2}\Big)S_{HL}(Q,\mu)^{\frac{2Q-\mu}{Q+2-\mu}},
\end{equation}
{there exists a family of bubbles $\{\mathfrak{g}_{\lambda_i,\xi_i}U_\mu\}_{i=1}^{m}$ such that}
\begin{equation*}
\left\|u-\sum\limits_{i=1}^m\mathfrak{g}_{\lambda_i,\xi_i}U_\mu\right\|_{S^{1,2}(\mathbb{H}^{n})} \leq C \left\|-\Delta_{\mathbb{H}}u-\left(\int_{\mathbb{H}^{n}}\frac{|u(\eta)|^{Q^{\ast}_{\mu}}}{|\eta^{-1}\xi|^{\mu}}{d}\eta\right)|u|^{Q^{\ast}_{\mu}-2}u\right\|_{(S^{1,2}(\mathbb{H}^{n}))^{-1}}.
\end{equation*}
Analogous to the variational characterization in \eqref{defH}, we consider the following minimization problem in the critical-point setting:
\begin{equation*}
 {H_{CP}(m)=\inf\limits_{\substack{u\in S^{1,2}(\mathbb{H}^{n})\setminus \mathfrak{M}_m,\ u\ge0,\\ u~\mathrm{satisfies}~\eqref{addbound}}}\frac{\Big\|-\Delta_{\mathbb{H}}u-\Big(\int_{\mathbb{H}^{n}}\frac{|u(\eta)|^{Q^{\ast}_{\mu}}}{|\eta^{-1}\xi|^{\mu}}{d}\eta\Big)
 |u|^{Q^{\ast}_{\mu}-2}u\Big\|_{(S^{1,2}(\mathbb{H}^{n}))^{-1}}}{\mathrm{dist}(u,\mathfrak{M}_m)}.}
\end{equation*}
{Here
\[
\mathfrak{M}_m:=\left\{\sum_{i=1}^{m}\mathfrak{g}_{\lambda_i,\xi_i}U_\mu:\lambda_i>0,\ \xi_i\in\mathbb{H}^{n}\right\}
\]
denotes the set of $m$-bubble profiles.} {Motivated by the work of De Nitti and K\"{o}nig \cite{DK23}, we study the corresponding optimal residual quotient in the single-bubble case $m=1$ and derive its local spectral upper bound.}

\begin{theorem}\label{addthm}
\begingroup
Let $Q\ge4$ and $0<\mu<Q$. Then
\begin{equation}\label{CP-nonstrict}
H_{CP}(1)\le H_{NS}^{\mathrm{spec}}
=\frac{(4n+8-2\mu)(\mu+4)}{2(n+4)(4n+8-\mu)}.
\end{equation}
\endgroup
\end{theorem}

\noindent{\bf Main novelties and technical points.}
The results above should be viewed in comparison with the Euclidean nonlocal stability result \cite{Z26} and the local Heisenberg result \cite{TZZ24}. The main novelty of the present paper lies in the quantitative interaction between the nonlocal HLS term and the CR geometry, which leads to new spectral and nonlinear features that are absent in either of these two settings. We emphasize the following points.

\vspace{.2cm}

{\bf (A) A coupled CR--HLS spectral structure.}
For the local Folland--Stein--Sobolev deficit, the second variation is governed by a single spectral form. In the present nonlocal problem, the linearization of the HLS interaction produces an additional bilinear form, so that the local stability analysis becomes a genuinely coupled spectral problem. After the Cayley transform, we simultaneously diagonalize the two forms on the complex spherical harmonics $\mathcal H_{i,j}^{n+1}$ and identify the common lowest normal mode. This yields the explicit threshold
\[
H_{NS}^{\mathrm{spec}}
=\frac{(4n+8-2\mu)(\mu+4)}
{2(n+4)(4n+8-\mu)},
\]
which reflects the interaction between the CR energy and the nonlocal HLS spectrum.

\vspace{.1cm}

{\bf (B) A nonlocal fourth-order strictness mechanism.}
\begingroup
The quadratic expansion gives only the local spectral threshold, and the cubic contribution vanishes on the lowest real normal eigenspace. We therefore introduce a second-order normal correction and expand the double-integral functional to fourth order, following the normal-form mechanism of \cite[Proposition~6.4]{Li26}. The nonlocal HLS term produces mode-dependent multipliers, so both the optimal correction and the fourth-order coefficient must be recomputed. The resulting finite-dimensional minimization yields the explicit quantity $\Gamma_{n,\mu}$; the sign condition $\Gamma_{n,\mu}<0$ gives a strict improvement over the spectral threshold. We verify this condition by an explicit algebraic calculation at $\mu=2$ for every $n\ge1$. A separate two-bubble interaction expansion provides the second strict upper bound needed for compactness.
\endgroup

\vspace{.1cm}

{\bf (C) Compactness and a strict local--nonlocal comparison.}
{The compactness argument uses two different barriers. In the first parameter regime of Theorem~\ref{main thm1}, the strict spectral bound excludes minimizing sequences approaching the extremal manifold; in the large-$\mu$ regime, the local Sobolev stability threshold $2/(n+4)$ plays the same role. In both regimes, the strict two-bubble bound rules out dichotomy and splitting into several nontrivial profiles. Combined with the profile decomposition on $\mathbb H^n$, these estimates yield attainment of $H_{NS}$. For these parameters, the minimizer together with the strict factorization through the sharp HLS inequality away from $\mathfrak M$ gives}
\[
H_{NS}>H_{BE}.
\]
Thus the nonlocal interaction genuinely changes the optimal lower stability constant, even though the sharp universal upper defect bound remains equal to $1$.

\vspace{.1cm}

{\bf (D) Stability for critical points.}
Besides the functional-deficit quotient, we also study the stability quotient defined through the Euler--Lagrange residual in $(S^{1,2})^{-1}$. Using the same coupled normal spectral decomposition, we obtain the universal estimate
\[
H_{CP}(1)\le H_{NS}^{\mathrm{spec}}.
\]
This shows that the spectral structure governing the local functional stability also controls the corresponding critical-point stability problem.

\vspace{.2cm}

\noindent{\bf Structure of the paper.}
\begingroup
Section~\ref{sec3} diagonalizes the two quadratic forms generated by the nonlocal linearization and derives the local spectral threshold. Section~\ref{sec4} proves the spectral strictness criterion and the strict two-bubble bound in Theorem~\ref{main thm0}. These estimates provide the compactness input for the minimizer argument in Section~\ref{sec5}. Section~\ref{sec6} determines the sharp universal upper defect constant, and Section~\ref{sec7} proves the local spectral upper bound for the critical-point residual quotient. Finally, Section~\ref{App} contains the explicit verification of $\Gamma_{n,2}<0$. Throughout the paper, $C$ denotes a positive constant whose value may change from line to line.
\endgroup

\section{Coupled spectral problems}\label{sec3}
{The second variation of the nonlocal deficit contains two distinct quadratic forms. We diagonalize them separately and then combine their spectra. The first eigenvalue problem is}

\begin{equation*}
 -\Delta_{\mathbb{H}} \varphi=\nu\left(\int_{\mathbb{H}^{n}}\frac{|U_\mu(\eta)|
^{Q^{\ast}_{\mu}}}{|\eta^{-1}\xi|^{\mu}}{d}\eta\right)|U_\mu|^{Q^{\ast}_{\mu}-2}\varphi,~~\xi\in\mathbb{H}^{n},~~\varphi\in S^{1,2}(\mathbb{H}^{n})\setminus \{0\},
\end{equation*}
which is equivalent to
\begin{equation*}
 -\Delta_{\mathbb{H}} \varphi=\nu U^{Q^{\ast}-2}\varphi,\quad\mathrm{in}~\mathbb{H}^{n},~~\varphi\in S^{1,2}(\mathbb{H}^{n})\setminus \{0\}.
\end{equation*}
From \cite[Section 5.4]{FL12}, it is not difficult to verify that

\begin{equation*}
  \nu_{i,j}=\frac{(Q-2+4i)(Q-2+4j)}{(Q-2)^2}.
\end{equation*}
Then
\begin{equation}\label{eig1}
 \nu_{0,0}=1<\nu_{1,0}=\frac{Q+2}{Q-2}<\nu_{2,0}=\frac{Q+6}{Q-2}<\nu_{1,1}=\frac{(Q+2)^2}{(Q-2)^2}<\cdots
\end{equation}
{After complexification, the eigenspace corresponding to $\nu_{i,j}$ is
\begin{equation*}
 \mathcal E_{i,j}^{\mathbb C}:=\mathcal C^*\mathcal H_{i,j}^{n+1}
 =\mathrm{span}_{\mathbb C}\left\{[\mathcal{J}_\mathcal{C}(\xi)]^{\frac{1}{Q^{\ast}}}\omega_{i,j}^k(\mathcal{C}\xi):
 k=1,\ldots,\mathrm{dim}\mathcal H_{i,j}^{n+1}\right\},
 \qquad i,j=0,1,2,\ldots,
\end{equation*}
where $\{\omega_{i,j}^k\}_{k=1}^{\mathrm{dim}\mathcal H_{i,j}^{n+1}}$ is a complex basis of $\mathcal H_{i,j}^{n+1}$. On the real function space, the corresponding eigenspaces are understood through the real blocks $(\mathcal H_{i,i}^{n+1})_{\mathbb R}$ and $(\mathcal H_{i,j}^{n+1}\oplus\mathcal H_{j,i}^{n+1})_{\mathbb R}$ for $i\ne j$, pulled back by $\mathcal C^*$.}

Next, let us consider another eigenvalue problem
\begin{equation}\label{eigtwo}
 -\Delta_{\mathbb{H}} \psi=\tau\left(\int_{\mathbb{H}^{n}}\frac{|U_\mu(\eta)|
^{Q^{\ast}_{\mu}-1}\psi(\eta)}{|\eta^{-1}\xi|^{\mu}}{d}\eta\right)|U_\mu|^{Q^{\ast}_{\mu}-1},~~\xi\in\mathbb{H}^{n},~~ \psi\in S^{1,2}(\mathbb{H}^{n})\setminus \{0\}.
\end{equation}
Since
\begin{align*}
 \int_{\mathbb{S}^{2n+1}}|\mathcal{C}_*\psi(\zeta)|^2 d \sigma(\zeta)=&\int_{\mathbb{H}^n}|\psi(\xi)|^2
 [\mathcal{J}_\mathcal{C}(\xi)]^{1-\frac{2}{Q^{\ast}}}d \xi\\ \leq & C \int_{\mathbb{H}^n}|\psi(\xi)|^2
 \frac{1}{(1+|z|^2)^2+t^2}d \xi\\
 \leq & C \bigg(\int_{\mathbb{H}^n}|\psi(\xi)|^{Q^{\ast}}d \xi\bigg)^{\frac{2}{Q^{\ast}}}
 \left(\int_{\mathbb{H}^n}\frac{1}{[(1+|z|^2)^2+t^2]^{\frac{Q}{2}}}d \xi\right)^{\frac{Q^{\ast}-2}{Q^{\ast}}}\\
 \leq & C \int_{\mathbb{H}^{n}}|\nabla_{\mathbb{H}}\psi|^{2}{d}\xi<+\infty,
\end{align*}
{the argument of \cite[Lemma 4.5]{YZ251}, with only notational changes, shows that \eqref{eigtwo} is equivalent to}
\begin{equation}\label{eigtwo1}
 \Psi(\zeta)= \mathcal{C}_*\psi(\zeta)=\tau G(Q)\alpha(Q,\mu)2^{\frac{-3Q+\mu+2}{2}}\int_{\mathbb{S}^{2n+1}}\int_{\mathbb{S}^{2n+1}}
 \frac{1}{|1-\zeta\cdot \overline{\zeta'}|^{\frac{Q-2}{2}}}\frac{1}{|1-\zeta'\cdot \overline{\zeta''}|^{\frac{\mu}{2}}}\Psi(\zeta'')d \sigma(\zeta'')d \sigma(\zeta'),
\end{equation}
where
\begin{equation*}
 G(Q)=\frac{2^{n-4}\Gamma^2(\frac{n}{2})}{\pi^{n+1}},\quad \alpha(Q,\mu)=S(Q)^{-\frac{Q-\mu}{2}}C(Q,\mu)^{-1}(2n)^{Q+2-\mu}
 =\frac{2^{2n+2-\frac{\mu}{2}}n^2\Gamma^2\big(n+1-\frac{\mu}{4}\big)}{\pi^{n+1}\Gamma\big(n+1-\frac{\mu}{2}\big)
 }.
\end{equation*}
{Since $\Psi=\mathcal{C}_*\psi\in L^2(\mathbb{S}^{2n+1};\mathbb R)\setminus\{0\}\subset L^2(\mathbb{S}^{2n+1};\mathbb C)$, it has the orthogonal expansion}
\begin{equation*}
 \Psi(\zeta)=\sum\limits_{i,j\geq0}\sum\limits_{k=1}^{\mathrm{dim} \mathcal{H}_{i,j}^{n+1}}\Psi_{i,j}^k Y_{i,j}^k(\zeta),
\end{equation*}
{where $\Psi_{i,j}^k=\int_{\mathbb{S}^{2n+1}}\Psi\,\overline{Y_{i,j}^k}d\sigma$, and at least one coefficient is nonzero. For every index with $\Psi_{i,j}^k\neq0$, combining \eqref{FHA} with \eqref{eigtwo1} gives}
\begin{equation*}
 \Psi_{i,j}^k=\tau G(Q)\alpha(Q,\mu)2^{\frac{-3Q+\mu+2}{2}}E_{i,j}(Q-2)E_{i,j}(\mu) \Psi_{i,j}^k,
\end{equation*}
which is equivalent to
\begin{equation*}
 \tau G(Q)\alpha(Q,\mu)2^{\frac{-3Q+\mu+2}{2}}E_{i,j}(Q-2)E_{i,j}(\mu)=1.
\end{equation*}
Therefore,
\begin{align}\label{eig2}
 \tau_{0,0}=1<\tau_{1,0}=\frac{(n+2)(4n+4-\mu)}{n\mu}<&\tau_{2,0}=\frac{(n+4)(4n+4-\mu)(4n+8-\mu)}{n\mu(\mu+4)}
\nonumber \\ <&\tau_{1,1}=\frac{(n+2)^2(4n+4-\mu)^2}{n^2\mu^2}<\cdots
\end{align}
{The second spectral problem has the same complexified eigenspaces $\mathcal E_{i,j}^{\mathbb C}=\mathcal C^*\mathcal H_{i,j}^{n+1}$ as the first one; only the eigenvalues change from $\nu_{i,j}$ to $\tau_{i,j}$. Accordingly, on the real function space we again use the real blocks $(\mathcal H_{i,i}^{n+1})_{\mathbb R}$ and $(\mathcal H_{i,j}^{n+1}\oplus\mathcal H_{j,i}^{n+1})_{\mathbb R}$, pulled back by $\mathcal C^*$.}

{For $\mathfrak{M}=\{c\mathfrak{g}_{\lambda,\xi}U:c\in\mathbb{R},\lambda>0,\xi\in\mathbb{H}^{n}\}$, the tangent space at $U$ is}
\begin{align*}
 T_U \mathfrak{M}=&{\rm span}\left\{U, \frac{\partial \mathfrak{g}_{r,0}U}{\partial r}\bigg|_{r=1},\frac{\partial \mathfrak{g}_{1,\eta}U}{\partial \eta^{(a)}}\bigg|_{\eta=0},~a=1,\ldots,2n+1\right\}\\
 =&\mathcal C^*\left[(\mathcal H_{0,0}^{n+1})_{\mathbb R}
 \oplus(\mathcal H_{1,0}^{n+1}\oplus\mathcal H_{0,1}^{n+1})_{\mathbb R}\right].
\end{align*}
Therefore, from the above analysis, we obtain the following spectral gap inequalities.
\begin{proposition}\label{sg}
{Let $\phi\in(T_U\mathfrak{M})^\perp$. Then}
\begin{equation}\label{sg1}
 \int_{\mathbb{H}^{n}}|\nabla_\mathbb{H} \phi|^2 d\xi\geq \nu_{2,0}\int_{\mathbb{H}^{n}}\frac{U_\mu
^{Q^{\ast}_{\mu}}(\xi)U_\mu^{Q^{\ast}_{\mu}-2}(\eta)\phi^2(\eta)}{|\eta^{-1}\xi|^{\mu}}{d}\xi d\eta,
\end{equation}
and
\begin{equation}\label{sg2}
 \int_{\mathbb{H}^{n}}|\nabla_\mathbb{H} \phi|^2 d\xi\geq \tau_{2,0}\int_{\mathbb{H}^{n}}\frac{U_\mu
^{Q^{\ast}_{\mu}-1}(\xi)\phi(\xi) U_\mu^{Q^{\ast}_{\mu}-1}(\eta)\phi(\eta)}{|\eta^{-1}\xi|^{\mu}}{d}\xi d\eta,
\end{equation}
{where $\nu_{2,0}$ and $\tau_{2,0}$ are given by \eqref{eig1} and \eqref{eig2}, respectively.} {Moreover, equality in either \eqref{sg1} or \eqref{sg2} holds on the real lowest normal eigenspace corresponding, after the Cayley transform, to
\[
(\mathcal H_{2,0}^{n+1}\oplus\mathcal H_{0,2}^{n+1})_{\mathbb R}.
\]
In particular, a convenient real eigenfunction is
\begin{equation*}
\phi(\xi)=[\mathcal J_{\mathcal C}(\xi)]^{\frac{1}{Q^{\ast}}}\operatorname{Re}\!\left((\mathcal C\xi)_{n+1}^{2}\right).
\end{equation*}
This choice is valid for every $n\ge1$.}
\end{proposition}

{We next record the local lower bound for sequences approaching the extremal manifold. This is the compactness threshold used in Section~\ref{sec5}.}
\begin{lemma}\label{local}
Let $Q\geq 4$, $0<\mu< Q$ with $\mu\leq 4$. For any sequence $\{u_k\}\subset S^{1,2}(\mathbb{H}^{n})\setminus \mathfrak{M}$ satisfying
\begin{equation*}
 \inf\limits_k \|u_k\|_{S^{1,2}(\mathbb{H}^{n})}>0,\quad \mathrm{dist}(u_k,\mathfrak{M})\rightarrow0,
\end{equation*}
we have
\begin{equation}\label{local1}
 \liminf\limits_{k\rightarrow+\infty}\mathcal{L}(u_k)\geq 1-\left(\frac{Q^{\ast}_{\mu}-1}{\nu_{2,0}}+\frac{Q^{\ast}_{\mu}}{\tau_{2,0}}\right),
\end{equation}
{where $\nu_{2,0}$ and $\tau_{2,0}$ are given by \eqref{eig1} and \eqref{eig2}, respectively.}
\end{lemma}
\begin{proof}
{A direct computation gives}
\begin{equation*}
 1-\left(\frac{Q^{\ast}_{\mu}-1}{\nu_{2,0}}+\frac{Q^{\ast}_{\mu}}{\tau_{2,0}}\right)=\frac{(4n+8-2\mu)(\mu+4)}{2(n+4)(4n+8-\mu)}>\frac{2}{n+4}
\end{equation*}
for all $\mu\in (0,Q)$.
{Set $d_k={\rm dist}(u_k,\mathfrak{M})=\inf_{c\in \mathbb{R},\lambda>0,\xi\in \mathbb{H}^n}\|\nabla_\mathbb{H}(u_k-c\mathfrak{g}_{\lambda,\xi}U)\|_{L^{2}(\mathbb{H}^{n})}$. Then $d_k\to0$ as $k\to+\infty$.}
Since $d_k\to0$ and $\inf_k\|u_k\|_{S^{1,2}}>0$, the coefficient of a closest point cannot tend to zero. Thus, for $k$ large, the closest point lies in the smooth part $\mathfrak{M}^*=\mathfrak{M}\setminus \{0\}$. By the local projection theorem onto a finite-dimensional $C^2$ submanifold of a Hilbert space, after passing to a subsequence, there exists $(c_k,\lambda_k,\xi_k)\in \mathbb{R}\setminus \{0\}\times\mathbb{R}^+\times\mathbb{H}^n$ such that
\begin{equation*}
 d_k=\big\|\nabla_\mathbb{H}(u_k-c_k\mathfrak{g}_{\lambda_k,\xi_k}U)\big\|_{{L}^{2}(\mathbb{H}^n)}.
\end{equation*}
{Since $\mathfrak M^*$ is a $(2n+3)$-dimensional smooth manifold embedded in $S^{1,2}(\mathbb H^n)$,}
we have
 \begin{equation*}
 \big(u_k-c_k\mathfrak{g}_{\lambda_k,\xi_k}U\big)\perp T_{c_k\mathfrak{g}_{\lambda_k,\xi_k}U}\mathfrak{M},
 \end{equation*}
 where the tangent space at $(c_k,\lambda_k,\xi_k)$ is given by
 \begin{equation*}
 T_{c_k\mathfrak{g}_{\lambda_k,\xi_k}U}\mathfrak{M}={\rm span}\left\{\mathfrak{g}_{\lambda_k,\xi_k}U, \frac{\partial \mathfrak{g}_{r,\xi_k}U}{\partial r}\bigg|_{r=\lambda_k},\frac{\partial \mathfrak{g}_{\lambda_k,\eta}U}{\partial \eta^{(a)}}\bigg|_{\eta=\xi_k},~a=1,\ldots,2n+1
 \right\}.
 \end{equation*}
 Let
 \begin{equation*}
 u_k=c_k\mathfrak{g}_{\lambda_k,\xi_k}U+d_k w_k.
 \end{equation*}
 {Then $w_k$ is perpendicular to $T_{c_k\mathfrak{g}_{\lambda_k,\xi_k}U}\mathfrak{M}$,} $\|\nabla_\mathbb{H}w_k\|^2_{L^{2}(\mathbb{H}^n)}=1$ and
 \begin{equation*}
 \|\nabla_\mathbb{H}u_k\|^2_{L^{2}(\mathbb{H}^n)}=d_k^2\|\nabla_\mathbb{H}w_k\|
 _{L^{2}(\mathbb{H}^n)}^2+c_k^2\|\nabla_\mathbb{H}\mathfrak{g}_{\lambda_k,\xi_k}U\|^2_
 {L^{2}(\mathbb{H}^n)}
 =d_k^2+c_k^2\|\nabla_\mathbb{H}U\|^2_{L^{2}(\mathbb{H}^n)},
 \end{equation*}
 where we used
 \begin{equation*}
 \|\nabla_\mathbb{H}\mathfrak{g}_{\lambda_k,\xi_k}U\|_{L^{2}(\mathbb{H}^n)}=\|\nabla_\mathbb{H}U\|
 _{L^{2}(\mathbb{H}^n)}.
 \end{equation*}
 {Since $Q^{\ast}_{\mu}\geq2$, the preceding orthogonality gives}
 \begin{align}\label{local3}
 \int_{\mathbb{H}^{n}}\int_{\mathbb{H}^{n}}\frac{|u_k|^{Q^{\ast}_{\mu}}|u_k|
^{Q^{\ast}_{\mu}}}{|\eta^{-1}\xi|^{\mu}}&{d}\xi{d}\eta
 = c_k^{2\cdot Q_{\mu}^{\ast}} \int_{\mathbb{H}^{n}}\int_{\mathbb{H}^{n}}\frac{|\mathfrak{g}_{\lambda_k,\xi_k}U|^{Q^{\ast}_{\mu}}|\mathfrak{g}_{\lambda_k,\xi_k}U|
^{Q^{\ast}_{\mu}}}{|\eta^{-1}\xi|^{\mu}}{d}\xi{d}\eta
\nonumber \\& + Q_{\mu}^{\ast}(Q_{\mu}^{\ast}-1)c_k^{2(Q_{\mu}^{\ast}-1)}d_k^2 {\int_{\mathbb{H}^{n}}\int_{\mathbb{H}^{n}}\frac{|\mathfrak{g}_{\lambda_k,\xi_k}U|^{Q^{\ast}_{\mu}}|\mathfrak{g}_{\lambda_k,\xi_k}U|
^{Q^{\ast}_{\mu}-2} w_k^2}{|\eta^{-1}\xi|^{\mu}}{d}\xi{d}\eta} \nonumber\\
 & + (Q_{\mu}^{\ast})^2c_k^{2(Q_{\mu}^{\ast}-1)}d_k^2 \int_{\mathbb{H}^{n}}\int_{\mathbb{H}^{n}}\frac{|\mathfrak{g}_{\lambda_k,\xi_k}U|^{Q^{\ast}_{\mu}-1}w_k|\mathfrak{g}_{\lambda_k,\xi_k}U|
^{Q^{\ast}_{\mu}-1}w_k}{|\eta^{-1}\xi|^{\mu}}{d}\xi{d}\eta+ o(d_k^2),
 \end{align}
 since
 \begin{equation*}
 \int_{\mathbb{H}^{n}}\int_{\mathbb{H}^{n}}\frac{|\mathfrak{g}_{\lambda_k,\xi_k}U|^{Q^{\ast}_{\mu}}|\mathfrak{g}_{\lambda_k,\xi_k}U|
^{Q^{\ast}_{\mu}-1}w_k}{|\eta^{-1}\xi|^{\mu}}{d}\xi{d}\eta=S(Q)^{\frac{Q-\mu}{2}}C(Q,\mu)\int_{\mathbb{H}^{n}} \nabla_\mathbb{H}\mathfrak{g}_{\lambda_k,\xi_k}U\cdot \nabla_\mathbb{H}w_k d\xi=0.
 \end{equation*}
By \eqref{folland-stein} and \eqref{eq:HLSH}, using the
 H\"{o}lder inequality, we obtain
 \begin{align*}
 &\int_{\mathbb{H}^{n}}\int_{\mathbb{H}^{n}}\frac{|\mathfrak{g}_{\lambda_k,\xi_k}U|^{Q^{\ast}_{\mu}}|\mathfrak{g}_{\lambda_k,\xi_k}U|
^{Q^{\ast}_{\mu}}}{|\eta^{-1}\xi|^{\mu}}{d}\xi{d}\eta \\ \leq & C(Q,\mu)\|\mathfrak{g}_{\lambda_k,\xi_k}U\|_{L^{Q^{\ast}}(\mathbb{H}^{n})}^{2\cdot Q^{\ast}_{\mu}}=
C(Q,\mu)\|U\|_{L^{Q^{\ast}}(\mathbb{H}^{n})}^{2\cdot Q^{\ast}_{\mu}}\\
=&C(Q,\mu)S(Q)^{\frac{Q+2-\mu}{2}}\|U\|_{L^{Q^{\ast}}(\mathbb{H}^{n})}^{2}\leq S(Q)^{\frac{Q-\mu}{2}}C(Q,\mu)\|\nabla_\mathbb{H}U\|^2_{L^{2}(\mathbb{H}^{n})},
 \end{align*}
{where we used}
 \begin{equation*}
 \|\mathfrak{g}_{\lambda_k,\xi_k}U\|_{L^{Q^{\ast}}(\mathbb{H}^{n})}=\|U\|_{L^{Q^{\ast}}(\mathbb{H}^{n})}.
 \end{equation*}
From Proposition \ref{sg}, by a simple scaling argument, we have
\begin{equation*}
 \int_{\mathbb{H}^{n}}\int_{\mathbb{H}^{n}}\frac{|\mathfrak{g}_{\lambda_k,\xi_k}U_\mu|^{Q^{\ast}_{\mu}}|\mathfrak{g}_{\lambda_k,\xi_k}U_\mu|
^{Q^{\ast}_{\mu}-2} w_k^2}{|\eta^{-1}\xi|^{\mu}}{d}\xi{d}\eta\leq\frac{1}{\nu_{2,0}},
\end{equation*}
and
\begin{equation*}
 \int_{\mathbb{H}^{n}}\int_{\mathbb{H}^{n}}\frac{|\mathfrak{g}_{\lambda_k,\xi_k}U_\mu|^{Q^{\ast}_{\mu}-1}w_k|\mathfrak{g}_{\lambda_k,\xi_k}U_\mu|
^{Q^{\ast}_{\mu}-1} w_k}{|\eta^{-1}\xi|^{\mu}}{d}\xi{d}\eta\leq\frac{1}{\tau_{2,0}}.
\end{equation*}
{Since $U_\mu=\mathcal AU$, where $\mathcal A=S(Q)^{\frac{(Q-\mu)(2-Q)}{4(Q+2-\mu)}}C(Q,\mu)^{\frac{2-Q}{2(Q+2-\mu)}}$, these inequalities are equivalent to}
\begin{equation*}
 \int_{\mathbb{H}^{n}}\int_{\mathbb{H}^{n}}\frac{|\mathfrak{g}_{\lambda_k,\xi_k}U|^{Q^{\ast}_{\mu}}|\mathfrak{g}_{\lambda_k,\xi_k}U|
^{Q^{\ast}_{\mu}-2} w_k^2}{|\eta^{-1}\xi|^{\mu}}{d}\xi{d}\eta\leq\frac{1}{\mathcal{A}^{2(Q^{\ast}_{\mu}-1)}\nu_{2,0}},
\end{equation*}
and
\begin{equation*}
 \int_{\mathbb{H}^{n}}\int_{\mathbb{H}^{n}}\frac{|\mathfrak{g}_{\lambda_k,\xi_k}U|^{Q^{\ast}_{\mu}-1}w_k|\mathfrak{g}_{\lambda_k,\xi_k}U|
^{Q^{\ast}_{\mu}-1} w_k}{|\eta^{-1}\xi|^{\mu}}{d}\xi{d}\eta\leq\frac{1}{\mathcal{A}^{2(Q^{\ast}_{\mu}-1)}\tau_{2,0}}.
\end{equation*}
 Thus it follows from \eqref{local3} that
 \begin{align*}
 &\left(\int_{\mathbb{H}^{n}}\int_{\mathbb{H}^{n}}\frac{|u_k|^{Q^{\ast}_{\mu}}|u_k|
^{Q^{\ast}_{\mu}}}{|\eta^{-1}\xi|^{\mu}}{d}\xi{d}\eta\right)^{\frac{1}{Q^{\ast}_{\mu}}}\\
\leq& \left(c_k^{2\cdot Q_{\mu}^{\ast}} S(Q)^{\frac{Q-\mu}{2}}C(Q,\mu) \|\nabla_\mathbb{H}U\|^2_{L^{2}(\mathbb{H}^{n})}+\frac{Q_{\mu}^{\ast}(Q_{\mu}^{\ast}-1)c_k^{2(Q_{\mu}^{\ast}-1)}d_k^2}{\mathcal{A}^{2(Q^{\ast}_{\mu}-1)}\nu_{2,0}}
+\frac{(Q_{\mu}^{\ast})^2c_k^{2(Q_{\mu}^{\ast}-1)}d_k^2}{\mathcal{A}^{2(Q^{\ast}_{\mu}-1)}\tau_{2,0}}+o(d_k^2)\right)^{\frac{1}{Q^{\ast}_{\mu}}}\\
=& \mathcal{A}^{\frac{2}{Q^{\ast}_{\mu}}-2}
\left[c_k^{2} \|\nabla_\mathbb{H}U\|^{\frac{2}{Q^{\ast}_{\mu}}}_{L^{2}(\mathbb{H}^{n})}+\|\nabla_\mathbb{H}U\|^{\frac{2}{Q^{\ast}_{\mu}}-2}_{L^{2}(\mathbb{H}^{n})}
\left(\frac{Q^{\ast}_{\mu}-1}{\nu_{2,0}}+\frac{Q^{\ast}_{\mu}}{\tau_{2,0}}\right)d_k^2\right]+o(d_k^2).
 \end{align*}
{Therefore,}
 \begin{align*}
 & \int_{\mathbb{H}^{n}}|\nabla_{\mathbb{H}} u_k|^{2}{d}\xi-S_{HL}(Q,\mu) \left(\int_{\mathbb{H}^{n}}\int_{\mathbb{H}^{n}}\frac{|u_k|^{Q^{\ast}_{\mu}}|u_k|
^{Q^{\ast}_{\mu}}}{|\eta^{-1}\xi|^{\mu}}{d}\xi{d}\eta\right)^{\frac{1}{Q^{\ast}_{\mu}}}\\
\geq&
d_k^2\left[1-S_{HL}(Q,\mu)
\mathcal{A}^{\frac{2}{Q^{\ast}_{\mu}}-2} \|\nabla_\mathbb{H}U\|^{\frac{2}{Q^{\ast}_{\mu}}-2}_{L^{2}(\mathbb{H}^{n})}\left(\frac{Q^{\ast}_{\mu}-1}{\nu_{2,0}}+\frac{Q^{\ast}_{\mu}}{\tau_{2,0}}\right)\right]\\
&+c_k^2\left(\|\nabla_\mathbb{H}U\|^2_{{L}^{2}(\mathbb{H}^n)}-S_{HL}(Q,\mu)
\mathcal{A}^{\frac{2}{Q^{\ast}_{\mu}}-2} \|\nabla_\mathbb{H}U\|^{\frac{2}{Q^{\ast}_{\mu}}}_{L^{2}(\mathbb{H}^{n})}\right)+o(d_k^2)
\\
=&\left[1-\left(\frac{Q^{\ast}_{\mu}-1}{\nu_{2,0}}+\frac{Q^{\ast}_{\mu}}{\tau_{2,0}}\right)\right]d_k^2+o(d_k^2),
 \end{align*}
{where we used}
 \begin{align*}
 &S_{HL}(Q,\mu)\mathcal{A}^{\frac{2}{Q^{\ast}_{\mu}}-2} \|\nabla_\mathbb{H}U\|^{\frac{2}{Q^{\ast}_{\mu}}-2}_{L^{2}(\mathbb{H}^{n})}\\
 =&S(Q)C(Q,\mu)^{-\frac{1}{Q^{\ast}_{\mu}}}
 \left(S(Q)^{\frac{(Q-\mu)(2-Q)}{4(Q+2-\mu)}}C(Q,\mu)^{\frac{2-Q}{2(Q+2-\mu)}}\right)^{-\frac{2(Q+2-\mu)}{2Q-\mu}}
 S(Q)^{-\frac{Q(Q+2-\mu)}{2(2Q-\mu)}}=1.
 \end{align*}
{Dividing by $d_k^2=\mathrm{dist}(u_k,\mathfrak M)^2$ and letting $k\to\infty$ yields \eqref{local1}.}
\end{proof}

The preceding lemma identifies the limiting value of the stability quotient near the bubble manifold. The next section gives an explicit fourth-order criterion for improving this limiting value and complements it with a second threshold coming from bubble splitting.

\section{\texorpdfstring{{Upper bounds for $H_{NS}$: Proof of Theorem~\ref{main thm0}}}{Upper bounds for HNS: Proof of Theorem 3.1}}\label{sec4}

We use two genuinely different test configurations. A corrected perturbation of one bubble yields the spectral strictness criterion, whereas two separated bubbles give the unconditional dichotomy threshold relevant to concentration-compactness.

\begingroup
\subsection{Spectral upper bound}\label{sec41}
\begingroup
The cubic contribution vanishes along the lowest real normal eigenspace. Indeed, if \(P\in\mathcal H_{2,0}^{n+1}\), then
\[
P(e^{\mathrm i\vartheta}\zeta)=e^{2\mathrm i\vartheta}P(\zeta).
\]
Since \(d\sigma\) is \(U(1)\)-invariant,
\[
\int_{\mathbb S^{2n+1}}P^3d\sigma=0.
\]
More generally, if
\[
\phi\in(\mathcal H_{2,0}^{n+1}\oplus\mathcal H_{0,2}^{n+1})_{\mathbb R},
\]
then \(\phi=P+\overline P\) for a suitable \(P\in\mathcal H_{2,0}^{n+1}\), and the four terms in \(\phi^3\) have \(U(1)\)-weights \(6,2,-2,-6\). Hence
\[
\int_{\mathbb S^{2n+1}}\phi^3d\sigma=0.
\]
Thus a third-order expansion cannot produce strictness in this direction. We therefore introduce a second-order normal correction and expand the quotient to fourth order, following \cite[Proposition~6.4]{Li26}. Because the HLS term is nonlocal, all fourth-order interaction coefficients must be recomputed.
\endgroup

\begin{proof}[Proof of Theorem~\ref{main thm0}-\eqref{upper1}]
Put
\[
q:=Q^{\ast}_{\mu}=\frac{4n+4-\mu}{2n},\qquad
\nu_2:=\nu_{2,0}=\frac{n+4}{n},
\]
and use probability measure $d\sigma$ on $\mathbb S^{2n+1}$. After the Cayley transform and the amplitude normalization sending $U_\mu$ to $1$, let $\mathscr A$ be the normalized CR energy operator. {To keep the double-integral structure visible, we do not encode the HLS term by a one-variable operator. Instead, define the normalized symmetric bilinear form
\begin{equation*}
 \mathcal B_\mu(F,G):=\frac{1}{\mathcal N_\mu}
 \int_{\mathbb S^{2n+1}}\int_{\mathbb S^{2n+1}}
 \frac{F(\zeta)G(\eta)}{|1-\zeta\cdot\overline\eta|^{\frac{\mu}{2}}}d\sigma(\zeta)d\sigma(\eta),
 \qquad
 \mathcal N_\mu:=\int_{\mathbb S^{2n+1}}\int_{\mathbb S^{2n+1}}
 \frac{d\sigma(\zeta)d\sigma(\eta)}{|1-\zeta\cdot\overline\eta|^{\frac{\mu}{2}}}.
\end{equation*}
Thus $\mathcal B_\mu(1,1)=1$. By the Funk--Hecke formula, the integral operator associated with $\mathcal B_\mu$ acts on each complex spherical harmonic space $\mathcal H_{i,j}^{n+1}$ by the scalar
\begin{equation*}
 \vartheta_{i,j}:=\frac{E_{i,j}(\mu)}{E_{0,0}(\mu)}=\frac{\nu_{i,j}}{\tau_{i,j}}.
\end{equation*}
Consequently, on real functions the mutually orthogonal real blocks {$(\mathcal H_{i,i}^{n+1})_{\mathbb R}$} and $(\mathcal H_{i,j}^{n+1}\oplus\mathcal H_{j,i}^{n+1})_{\mathbb R}$ for $i\ne j$ diagonalize $\mathcal B_\mu$: the form vanishes between distinct blocks, while on a block indexed by $(i,j)$,
\begin{equation*}
 \mathcal B_\mu(F,G)=\vartheta_{i,j}\int_{\mathbb S^{2n+1}}FGd\sigma.
\end{equation*}}
On the bispherical harmonics,
\begin{equation*}
 \mathscr A|_{\mathcal H_{i,j}^{n+1}}=\nu_{i,j},
\end{equation*}
and $\mathscr A1=1$. {Up to the same common positive factor as in the original Heisenberg-group quotient, the spherical deficit is therefore
\begin{equation*}
 \mathscr D_\mu(v):=\langle v,\mathscr Av\rangle-
 \mathcal B_\mu(|v|^q,|v|^q)^{\frac{1}{q}}.
\end{equation*}}
Let $\mathfrak M_{\mathbb S}$ be the real optimizer cone and use the $\mathscr A$-metric for the distance.

Choose
\begin{equation}\label{phiLi}
 \phi(\zeta)=\operatorname{Re}(\zeta_{n+1}^2)
 \in(\mathcal H_{2,0}^{n+1}\oplus\mathcal H_{0,2}^{n+1})_{\mathbb R}.
\end{equation}
The elementary beta moments give
\begin{equation*}
 M_2:=\int_{\mathbb S^{2n+1}}\phi^2d\sigma=\frac1{(n+1)(n+2)},\qquad
 \int_{\mathbb S^{2n+1}}\phi^3d\sigma=0,
\end{equation*}
and
\begin{equation*}
 M_4:=\int_{\mathbb S^{2n+1}}\phi^4d\sigma=\frac9{(n+1)(n+2)(n+3)(n+4)}.
\end{equation*}
Apart from its constant component, $\phi^2$ has only the real components
$40$, $11$, and $22$. Write $P_\alpha$ for the corresponding orthogonal projections, $\mathcal I=\{40,11,22\}$, and
\begin{equation}\label{projection-masses}
\begin{aligned}
 m_{40}&:=\|P_{40}\phi^2\|_2^2=\frac3{(n+1)(n+2)(n+3)(n+4)},\\
 m_{11}&:=\|P_{11}\phi^2\|_2^2=\frac{4n}{(n+1)^2(n+2)(n+3)^2},\\
 m_{22}&:=\|P_{22}\phi^2\|_2^2=\frac{n}{(n+2)^2(n+3)^2(n+4)}.
\end{aligned}
\end{equation}

For brevity, we write
\begin{equation*}
  \nu_{40}:=\nu_{4,0},\quad \nu_{11}:=\nu_{1,1},\quad \nu_{22}:=\nu_{2,2},
\end{equation*}
and similarly for $\vartheta_\alpha$.

Set
\begin{equation*}
 \vartheta_2:=\vartheta_{2,0}=\frac{\mu(\mu+4)}{(4n+4-\mu)(4n+8-\mu)},\qquad
 b:=(q-1)+q\vartheta_2.
\end{equation*}
Then
\begin{equation*}
 \kappa_{\rm spec}:=1-\frac b{\nu_2}
 =1-\left(\frac{q-1}{\nu_{2,0}}+\frac q{\tau_{2,0}}\right)
 =H_{NS}^{\rm spec}.
\end{equation*}
For $\alpha\in\mathcal I$,
we have
\begin{equation}\label{higher-eigenvalues}
  \nu_{40}=\frac{n+8}{n},\quad \nu_{11}=\frac{(n+2)^2}{n^2}, \quad \nu_{22}=\frac{(n+4)^2}{n^2}.
\end{equation}
By the definition of $\vartheta_{i,j}$, we further obtain

\begin{equation*}
\begin{aligned}
 \vartheta_{40}=\frac{\mu(\mu+4)(\mu+8)(\mu+12)}{(4n+4-\mu)(4n+8-\mu)(4n+12-\mu)(4n+16-\mu)},\
 \vartheta_{11}=\left(\frac\mu{4n+4-\mu}\right)^2,\
 \vartheta_{22}=\vartheta_2^2.
\end{aligned}
\end{equation*}
It is useful to introduce the quadratic quotient at each normal level,
\begin{equation*}
 \kappa_\alpha:=1-\frac{(q-1)+q\vartheta_\alpha}{\nu_\alpha}.
\end{equation*}
{Since
\[
\kappa_\alpha=1-\left(\frac{q-1}{\nu_\alpha}+\frac{q}{\tau_\alpha}\right),
\]
and} \(\nu_\alpha>\nu_{2,0}\), \(\tau_\alpha>\tau_{2,0}\) for every \(\alpha\in\mathcal I\), we have $\kappa_\alpha>\kappa_{\rm spec}$.

Take
\begin{equation*}
 r_\epsilon=\epsilon\phi+\epsilon^2\psi,\qquad
 \psi=\sum_{\alpha\in\mathcal I}\psi_\alpha,\qquad
 \psi_\alpha\in P_\alpha L^2(\mathbb S^{2n+1}).
\end{equation*}
{The modes occurring in $\phi$ and $\psi$ are orthogonal to the tangent modes of the optimizer cone at $1$; hence $r_\epsilon$ lies in the normal space at $1$, and $\langle\phi,\psi\rangle=0$. By the tubular-neighborhood projection, the nearest point of $1+r_\epsilon$ on $\mathfrak M_{\mathbb S}$ is therefore $1$ for $|\epsilon|$ sufficiently small. Consequently,}
\begin{equation}\label{exact-distance-fourth}
 \operatorname{dist}_{\mathscr A}(1+r_\epsilon,\mathfrak M_{\mathbb S})^2
 =\nu_2M_2\epsilon^2+\epsilon^4\sum_{\alpha\in\mathcal I}\nu_\alpha\|\psi_\alpha\|_2^2.
\end{equation}

We now give the fourth-order calculation explicitly. Put
{\begin{equation*}
 J_2:=\mathcal B_\mu(\phi^2,\phi^2)
 =M_2^2+\sum_{\alpha\in\mathcal I}\vartheta_\alpha m_\alpha.
\end{equation*}}
By spectral orthogonality,
\[
\int_{\mathbb S^{2n+1}}\phi d\sigma
=\int_{\mathbb S^{2n+1}}\psi d\sigma
=\int_{\mathbb S^{2n+1}}\phi\psi d\sigma=0,
\]
while the $U(1)$-symmetry gives
\[
\int_{\mathbb S^{2n+1}}\phi^3 d\sigma=0.
\]
Since $\mathcal B_\mu(\phi,h)
=\vartheta_2\int_{\mathbb S^{2n+1}}\phi h d\sigma$, these identities also imply
$\mathcal B_\mu(\phi,\psi)=\mathcal B_\mu(\phi,\phi^2)=0$.
Consequently, the linear and cubic terms in the following expansion vanish.
A Taylor expansion at the positive constant $1$ gives
\begin{align}\label{I-fourth}
 &{\mathcal B_\mu((1+r_\epsilon)^q,(1+r_\epsilon)^q)}\nonumber\\
 =&1+qbM_2\epsilon^2+\epsilon^4\Bigg[
 q(q-1)\|\psi\|_2^2+q^2\sum_{\alpha\in\mathcal I}\vartheta_\alpha\|\psi_\alpha\|_2^2
 +q\sum_{\alpha\in\mathcal I}c_\alpha\langle P_\alpha\phi^2,\psi_\alpha\rangle\nonumber\\
 &\qquad+\frac{q(q-1)(q-2)(q-3)}{12}M_4
 +\frac{q^2(q-1)(q-2)}3\vartheta_2M_4
 +\frac{q^2(q-1)^2}{4}J_2\Bigg]+o(\epsilon^4),
\end{align}
where
\begin{equation}\label{dc'-def}
 c_\alpha:=(q-1)\big(q-2+2q\vartheta_2+q\vartheta_\alpha\big).
\end{equation}
Applying the outer $1/q$ power to \eqref{I-fourth} yields
\begin{align}\label{outer-fourth}
 &{\mathcal B_\mu((1+r_\epsilon)^q,(1+r_\epsilon)^q)}^{\frac{1}{q}}\nonumber\\
 =&1+bM_2\epsilon^2+\epsilon^4\Bigg[
 (q-1)\|\psi\|_2^2+q\sum_{\alpha\in\mathcal I}\vartheta_\alpha\|\psi_\alpha\|_2^2
 +\sum_{\alpha\in\mathcal I}c_\alpha\langle P_\alpha\phi^2,\psi_\alpha\rangle\nonumber\\
 &\qquad+\frac{(q-1)(q-2)(q-3)}{12}M_4
 +\frac{q(q-1)(q-2)}3\vartheta_2M_4
 +\frac{q(q-1)^2}{4}J_2
 -\frac{q-1}{2}b^2M_2^2\Bigg]+o(\epsilon^4).
\end{align}
On the other hand,
\begin{equation}\label{energy-fourth}
 \langle1+r_\epsilon,\mathscr A(1+r_\epsilon)\rangle
 =1+\nu_2M_2\epsilon^2+\epsilon^4\sum_{\alpha\in\mathcal I}\nu_\alpha\|\psi_\alpha\|_2^2.
\end{equation}
\begingroup
Combining \eqref{outer-fourth} with \eqref{energy-fourth} and then subtracting $\kappa_{\rm spec}$ times \eqref{exact-distance-fourth}, we obtain
\begin{equation*}
 \mathscr D_\mu(1+r_\epsilon)-\kappa_{\rm spec}\operatorname{dist}_{\mathscr A}(1+r_\epsilon,\mathfrak M_{\mathbb S})^2
 =\epsilon^4\mathscr C_\mu(\psi)+o(\epsilon^4),
\end{equation*}
where
\begin{equation*}
 \mathscr C_\mu(\psi)=\mathscr C_0+
 \sum_{\alpha\in\mathcal I}\left[d_\alpha\|\psi_\alpha\|_2^2-c_\alpha\langle P_\alpha\phi^2,\psi_\alpha\rangle\right].
\end{equation*}
Indeed, the coefficient of $\|\psi_\alpha\|_2^2$ is
\[
(1-\kappa_{\rm spec})\nu_\alpha-(q-1)-q\vartheta_\alpha
=\frac b{\nu_2}\nu_\alpha-(q-1)-q\vartheta_\alpha.
\]
Thus
\begin{equation}\label{dc-def}
 d_\alpha:=\frac b{\nu_2}\nu_\alpha-(q-1)-q\vartheta_\alpha
 =\nu_\alpha(\kappa_\alpha-\kappa_{\rm spec})>0,
\end{equation}
\endgroup
and
\begin{align}\label{C0-def}
 \mathscr C_0={}&-\frac{(q-1)(q-2)(q-3)}{12}M_4
 -\frac{q(q-1)(q-2)}3\vartheta_2M_4\nonumber\\
 &-\frac{q(q-1)^2}{4}\left(M_2^2+\sum_{\alpha\in\mathcal I}\vartheta_\alpha m_\alpha\right)
 +\frac{q-1}{2}b^2M_2^2.
\end{align}
\begingroup
Since $d_\alpha>0$, the fourth-order coefficient is minimized by completing the square:
\[
\mathscr C_\mu(\psi)
=\mathscr C_0-\sum_{\alpha\in\mathcal I}\frac{c_\alpha^2}{4d_\alpha}m_\alpha
+\sum_{\alpha\in\mathcal I}d_\alpha
\left\|\psi_\alpha-\frac{c_\alpha}{2d_\alpha}P_\alpha\phi^2\right\|_2^2.
\]
Hence the unique optimal correction is
\[
 \psi_*:=\sum_{\alpha\in\mathcal I}\frac{c_\alpha}{2d_\alpha}P_\alpha\phi^2,
\]
and
\begin{equation}\label{Gamma-def}
 \Gamma_{n,\mu}:=\mathscr C_\mu(\psi_*)
 =\mathscr C_0-\sum_{\alpha\in\mathcal I}\frac{c_\alpha^2}{4d_\alpha}m_\alpha.
\end{equation}
\endgroup
Consequently,
\begin{equation*}
 \frac{\mathscr D_\mu(1+r_\epsilon)}{\operatorname{dist}_{\mathscr A}(1+r_\epsilon,\mathfrak M_{\mathbb S})^2}
 =\kappa_{\rm spec}+\frac{\Gamma_{n,\mu}}{\nu_2M_2}\epsilon^2+o(\epsilon^2).
\end{equation*}
If $\Gamma_{n,\mu}<0$, then $\nu_2M_2>0$ and the last expansion is strictly smaller than $\kappa_{\rm spec}$ for all sufficiently small nonzero $\epsilon$. Since $\phi$ and $\psi_*$ are smooth on the sphere, $1+r_\epsilon>0$ for such $\epsilon$. Undoing the Cayley transform and the amplitude normalization multiplies the numerator and the squared distance by the same positive factor. Therefore
\[
H_{NS}<H_{NS}^{\rm spec}.
\]

\end{proof}

\endgroup

\subsection{Two-bubble upper bound}
We next prove \eqref{upper2}. The construction is inspired by \cite{K25}, but the decisive expansion here is the HLS-type cross interaction with kernel $|\eta^{-1}\xi|^{-\mu}$ under Heisenberg translations and dilations. We begin with the required notation. We denote the $\|\cdot\|_*$-normalized bubble on the Heisenberg group by

\begin{equation*}
 B(\xi)=\frac{C_{Q,\mu}}{[(1+|z|^{2})^{2}+t^{2}]^{\frac{Q-2}{4}}},
\end{equation*}
with $C_{Q,\mu}>0$ chosen such that
\begin{equation*}
\|B\|_*= \left(\int_{\mathbb{H}^{n}}\int_{\mathbb{H}^{n}}\frac{|B(\xi)|^{Q^{\ast}_{\mu}} |B(\eta)|
^{Q^{\ast}_{\mu}}}{|\eta^{-1}\xi|^{\mu}}{d}\xi{d}\eta\right)^{\frac{1}{2\cdot Q^{\ast}_{\mu}}}=1.
\end{equation*}
Denote also $B_{\lambda,\xi_0}(\xi)=\mathfrak{g}_{\lambda,\xi_0}B(\xi)$ for $\lambda>0$ and $\xi_0\in \mathbb{H}^n$, which satisfies
\begin{equation*}
 \|B_{\lambda,\xi_0}\|_*=1,\quad \|\nabla_\mathbb{H}B_{\lambda,\xi_0}\|_{{L}^{2}(\mathbb{H}^n)}^2=S_{HL}(Q,\mu),
\end{equation*}
\begin{equation}\label{ne1}
 -\Delta_{\mathbb{H}} B_{\lambda,\xi_0}=S_{HL}(Q,\mu)\left(\int_{\mathbb{H}^{n}}\frac{B_{\lambda,\xi_0}
^{Q^{\ast}_{\mu}}(\eta)}{|\eta^{-1}\xi|^{\mu}}{d}\eta\right)B_{\lambda,\xi_0}^{Q^{\ast}_{\mu}-1},
\end{equation}
and
\begin{equation*}
 \|B_{\lambda,\xi_0}\|^{Q^{\ast}}_{L^{Q^{\ast}}(\mathbb{H}^{n})}=C(Q,\mu)^{-\frac{Q^{\ast}}{2\cdot Q^{\ast}_{\mu}}},
\end{equation*}
\begin{equation}\label{ne2}
 -\Delta_{\mathbb{H}} B_{\lambda,\xi_0}=S_{HL}(Q,\mu)C(Q,\mu)^{\frac{Q^{\ast}}{2\cdot Q^{\ast}_{\mu}}}B_{\lambda,\xi_0}^{Q^{\ast}-1}.
\end{equation}
Let
\begin{equation*}
 \mathfrak{M}_1=\{B_{\lambda,\xi_0}:\lambda>0,\xi_0\in \mathbb{H}^n\}\subset \mathfrak{M}
\end{equation*}
be the submanifold of $\mathfrak{M}$ consisting of $\|\cdot\|_*$-normalized bubbles.
For $u\in S^{1,2}(\mathbb{H}^{n})$, define
\begin{equation*}
 \mathbf{m}(u)=\sup\limits_{v\in \mathfrak{M}_1}\Big(u,v^{Q^{\ast}-1}\Big)^2,
\end{equation*}
where
\begin{equation*}
 \Big(u,v^{Q^{\ast}-1}\Big)=\int_{\mathbb{H}^{n}} u v^{Q^{\ast}-1}d\xi
\end{equation*}
denotes the pairing between $L^{Q^{\ast}}(\mathbb{H}^{n})$ and its dual $(L^{Q^{\ast}}(\mathbb{H}^{n}))'$.
{The next lemma gives a convenient reformulation of $\mathrm{dist}(u,\mathfrak M)$ in terms of the optimization problem $\mathbf m(u)$; see also \cite[Lemma 2.2]{K25}.}
\begin{lemma}\label{l4.1}
{For $u\in S^{1,2}(\mathbb H^n)$, we have}
\begin{equation}\label{eq:l4.1}
 \mathrm{dist}(u,\mathfrak{M})^2=\|\nabla_\mathbb{H} u\|_{L^2(\mathbb{H}^n)}^2-S_{HL}(Q,\mu)C(Q,\mu)^{\frac{Q^{\ast}}{Q^{\ast}_{\mu}}}\mathbf{m}(u).
\end{equation}
{Moreover, $\mathrm{dist}(u,\mathfrak{M})$ is attained. The function
$S_{HL}(Q,\mu)C(Q,\mu)^{\frac{Q^{\ast}}{2\cdot Q^{\ast}_{\mu}}}(u,v^{Q^{\ast}-1})v
$
realizes this distance if and only if $v\in\mathfrak M_1$ maximizes $\mathbf m(u)$. In addition, $\mathbf m(u)>0$ for every nonzero $u\in S^{1,2}(\mathbb H^n)$.}
\end{lemma}
\begin{proof}
Recall that
\begin{equation*}
 \|\nabla_\mathbb{H}v\|_{{L}^{2}(\mathbb{H}^n)}^2=S_{HL}(Q,\mu),\quad -\Delta_{\mathbb{H}} v=S_{HL}(Q,\mu)C(Q,\mu)^{\frac{Q^{\ast}}{2\cdot Q^{\ast}_{\mu}}}v^{Q^{\ast}-1}.
\end{equation*}
For any $c\in \mathbb{R}$ and $v\in \mathfrak{M}_1$, we have
\begin{align*}
 \|\nabla_\mathbb{H} (u-cv)\|_{L^2(\mathbb{H}^n)}^2=&\|\nabla_\mathbb{H} u\|_{L^2(\mathbb{H}^n)}^2+
 c^2S_{HL}(Q,\mu)-2cS_{HL}(Q,\mu)C(Q,\mu)^{\frac{Q^{\ast}}{2\cdot Q^{\ast}_{\mu}}}\Big(u,v^{Q^{\ast}-1}\Big)\\
 =&\|\nabla_\mathbb{H} u\|_{L^2(\mathbb{H}^n)}^2-S_{HL}(Q,\mu)C(Q,\mu)^{\frac{Q^{\ast}}{ Q^{\ast}_{\mu}}}\Big(u,v^{Q^{\ast}-1}\Big)^2\\
 &+S_{HL}(Q,\mu)\left(c-C(Q,\mu)^{\frac{Q^{\ast}}{2\cdot Q^{\ast}_{\mu}}}\Big(u,v^{Q^{\ast}-1}\Big)\right)^2.
\end{align*}
Hence, we deduce that
\begin{equation*}
 \mathrm{dist}(u,\mathfrak{M})^2=\inf\limits_{v\in \mathfrak{M}_1}\inf\limits_{c\in \mathbb{R}}
 \|\nabla_\mathbb{H} (u-cv)\|_{L^2(\mathbb{H}^n)}^2=
 \|\nabla_\mathbb{H} u\|_{L^2(\mathbb{H}^n)}^2-S_{HL}(Q,\mu)C(Q,\mu)^{\frac{Q^{\ast}}{Q^{\ast}_{\mu}}}\sup\limits_{v\in \mathfrak{M}_1}\Big(u,v^{Q^{\ast}-1}\Big)^2.
\end{equation*}
The rest of the proof is similar to that of \cite[Lemma 2.2]{K25}, so we omit it.
\end{proof}

Notice that from the definition of $\mathbf{m}$, we have
\begin{equation}\label{e4.7}
 {\mathbf{m}(u)\leq C(Q,\mu)^{-\frac{Q^{\ast}}{Q^{\ast}_{\mu}}}\|u\|_*^2,\qquad \forall \, u\in S^{1,2}(\mathbb{H}^{n}),}
\end{equation}
and there is equality in \eqref{e4.7} if and only if $u\in \mathfrak{M}$. Indeed, for any $u\in S^{1,2}(\mathbb{H}^{n})$ and $v\in \mathfrak{M}_1$, by \eqref{ne1} and \eqref{ne2},
we have
\begin{equation*}
 \int_{\mathbb{H}^{n}} u v^{Q^{\ast}-1}d\xi=C(Q,\mu)^{-\frac{Q^{\ast}}{2\cdot Q^{\ast}_{\mu}}}
 \int_{\mathbb{H}^{n}}\int_{\mathbb{H}^{n}}
 \frac{v^{Q^{\ast}_{\mu}}(\xi)v^{Q^{\ast}_{\mu}-1}(\eta)u(\eta)}{|\eta^{-1}\xi|^{\mu}}d\xi{d}\eta.
\end{equation*}
Then, using
 the Cauchy-Schwarz inequality and the H\"{o}lder inequality, we derive that
\begin{align*}
 \left(\int_{\mathbb{H}^{n}} u v^{Q^{\ast}-1}d\xi\right)^2\leq& C(Q,\mu)^{-\frac{Q^{\ast}}{Q^{\ast}_{\mu}}}
 \left(\int_{\mathbb{H}^{n}}\int_{\mathbb{H}^{n}}
 \frac{v^{Q^{\ast}_{\mu}}v^{Q^{\ast}_{\mu}}}{|\eta^{-1}\xi|^{\mu}}d\xi{d}\eta\right)\left(\int_{\mathbb{H}^{n}}\int_{\mathbb{H}^{n}}
 \frac{v^{Q^{\ast}_{\mu}-1}uv^{Q^{\ast}_{\mu}-1}u}{|\eta^{-1}\xi|^{\mu}}d\xi{d}\eta\right)
 \\ \leq& C(Q,\mu)^{-\frac{Q^{\ast}}{Q^{\ast}_{\mu}}} \|v\|_*^{2\cdot Q^{\ast}_{\mu}}\|v\|_*^{2( Q^{\ast}_{\mu}-1)}\|u\|_*^{2}
 =C(Q,\mu)^{-\frac{Q^{\ast}}{Q^{\ast}_{\mu}}}\|u\|_*^{2},
\end{align*}
and equality holds if and only if $u=cv\in \mathfrak{M}$ for some $c\in \mathbb{R}$.

{We now prove the two-bubble upper bound for $H_{NS}$ by considering a sequence of test}
functions of the form
\begin{equation*}
 u_\lambda(\xi)=B(\xi)+B_\lambda(\xi),
\end{equation*}
where $B_\lambda(\xi)=B_{\lambda,\mathbf{0}}(\xi)$ and $\lambda\rightarrow 0^+$. {The following proposition records the asymptotic estimates needed for $\mathcal{L}(u_\lambda)$.}
\begin{proposition}\label{p4.1}
Let $c_0=B(\mathbf{0})\int_{\mathbb{H}^{n}} B^{Q^{\ast}-1}d\xi$.
As $\lambda\rightarrow0$, the following holds:

(i) $\int_{\mathbb{H}^{n}}|\nabla_{\mathbb{H}} u_\lambda|^{2}{d}\xi=2S_{HL}(Q,\mu)+2S_{HL}(Q,\mu)
c_0C(Q,\mu)^{\frac{Q^{\ast}}{2\cdot Q^{\ast}_{\mu}}}\lambda^{\frac{Q-2}{2}}+o\big(\lambda^{\frac{Q-2}{2}}\big)$;

(ii) $\mathrm{dist}(u_\lambda,\mathfrak{M})^2=S_{HL}(Q,\mu)+o\big(\lambda^{\frac{Q-2}{2}}\big)$;

{(iii) $\Big(\int_{\mathbb{H}^{n}}\int_{\mathbb{H}^{n}}\frac{|u_\lambda(\xi)|^{Q^{\ast}_{\mu}}|u_\lambda(\eta)|
^{Q^{\ast}_{\mu}}}{|\eta^{-1}\xi|^{\mu}}{d}\xi{d}\eta\Big)^{\frac{1}{Q^{\ast}_{\mu}}}\geq2^{\frac{1}{Q^{\ast}_{\mu}}}+
2^{\frac{1}{Q^{\ast}_{\mu}}+1}c_0C(Q,\mu)^{\frac{Q^{\ast}}{2\cdot Q^{\ast}_{\mu}}}\lambda^{\frac{Q-2}{2}}+o\big(\lambda^{\frac{Q-2}{2}}\big)$.}
\end{proposition}
\begin{proof}
By \eqref{ne2}, an argument similar to that in \cite[Proposition~3.1]{K25} gives $(i)$ and $(ii)$. Now, let us prove $(iii)$. Set
\begin{equation*}
 I_{\tau}(\xi)=I_{\tau}(z,t)=\left(\frac{\tau^2z}{-|z|^2+it},\frac{-\tau^4t}{|z|^4+t^2}\right)
\end{equation*}
be the Heisenberg inversion for some $\tau>0$. Using the Kor\'{a}nyi-Reimann formula; see \cite{P16},
\begin{equation*}
 d\big(I_{\tau}(\eta),I_{\tau}(\xi)\big)=\frac{\tau^2 d(\eta,\xi)}{|\eta||\xi|},
\end{equation*}
we obtain
\begin{align*}
 &\int_{B^c(\mathbf{0},\lambda^{-1/2})}\int_{\mathbb{H}^{n}}\frac{|u_\lambda(\xi)|^{Q^{\ast}_{\mu}}|u_\lambda(\eta)|
^{Q^{\ast}_{\mu}}}{|\eta^{-1}\xi|^{\mu}}{d}\xi{d}\eta\\
=&
\int_{B^c(\mathbf{0},\lambda^{-1/2})}\int_{\mathbb{H}^{n}}\frac{(B+B_\lambda)^{Q^{\ast}_{\mu}}(\xi)(B+B_\lambda)
^{Q^{\ast}_{\mu}}(\eta)}{|\eta^{-1}\xi|^{\mu}}{d}\xi{d}\eta\\
=&\int_{B(\mathbf{0},\lambda^{-1/2})}\int_{\mathbb{H}^{n}}\frac{(B+B_\lambda)^{Q^{\ast}_{\mu}}(I_{\lambda^{-1/2}}(\xi_1))(B+B_\lambda)
^{Q^{\ast}_{\mu}}(I_{\lambda^{-1/2}}(\xi_2))}{d\big(I_{\lambda^{-1/2}}(\xi_2),I_{\lambda^{-1/2}}(\xi_1)\big)^{\mu}}
\bigg(\frac{\lambda^{-1/2}}{|\xi_1|}\bigg)^{2Q}\bigg(\frac{\lambda^{-1/2}}{|\xi_2|}\bigg)^{2Q}{d}\xi_1{d}\xi_2\\
=&\int_{B(\mathbf{0},\lambda^{-1/2})}\int_{\mathbb{H}^{n}}\frac{(B+B_\lambda)^{Q^{\ast}_{\mu}}(I_{\lambda^{-1/2}}(\xi_1))(B+B_\lambda)
^{Q^{\ast}_{\mu}}(I_{\lambda^{-1/2}}(\xi_2))}{d(\xi_2,\xi_1)^{\mu}}
\bigg(\frac{\lambda^{-1/2}}{|\xi_1|}\bigg)^{2Q-\mu}\bigg(\frac{\lambda^{-1/2}}{|\xi_2|}\bigg)^{2Q-\mu}{d}\xi_1{d}\xi_2\\
=&\int_{B(\mathbf{0},\lambda^{-1/2})}\int_{\mathbb{H}^{n}}\frac{(B+B_\lambda)^{Q^{\ast}_{\mu}}(\xi_1)(B+B_\lambda)
^{Q^{\ast}_{\mu}}(\xi_2)}{|\xi_2^{-1}\xi_1|^{\mu}}{d}\xi_1{d}\xi_2\\
=&\int_{B(\mathbf{0},\lambda^{-1/2})}\int_{\mathbb{H}^{n}}\frac{|u_\lambda(\xi)|^{Q^{\ast}_{\mu}}|u_\lambda(\eta)|
^{Q^{\ast}_{\mu}}}{|\eta^{-1}\xi|^{\mu}}{d}\xi{d}\eta,
\end{align*}
where
\begin{equation*}
 \xi_1=(z_1,t_1),\quad \xi_2=(z_2,t_2),
\end{equation*}
such that
\begin{equation*}
 \xi= (z,t)=\left(\frac{\lambda^{-1}z_1}{-|z_1|^2+it_1},\frac{-\lambda^{-2}t_1}{|z_1|^4+t_1^2}\right),\quad
 \eta= (z',t')=\left(\frac{\lambda^{-1}z_2}{-|z_2|^2+it_2},\frac{-\lambda^{-2}t_2}{|z_2|^4+t_2^2}\right).
\end{equation*}
{Therefore,
\begin{align*}
 &\int_{\mathbb{H}^{n}}\int_{\mathbb{H}^{n}}\frac{|u_\lambda(\xi)|^{Q^{\ast}_{\mu}}|u_\lambda(\eta)|
 ^{Q^{\ast}_{\mu}}}{|\eta^{-1}\xi|^{\mu}}{d}\xi{d}\eta\\
 =&2\int_{B(\mathbf{0},\lambda^{-1/2})}\int_{\mathbb{H}^{n}}\frac{(B+B_\lambda)^{Q^{\ast}_{\mu}}(\xi)
 (B+B_\lambda)^{Q^{\ast}_{\mu}}(\eta)}{|\eta^{-1}\xi|^{\mu}}{d}\xi{d}\eta\\
 \geq&2\int_{B(\mathbf{0},\lambda^{-1/2})}\int_{B(\mathbf{0},\lambda^{-1/2})}\frac{(B+B_\lambda)^{Q^{\ast}_{\mu}}(\xi)
 (B+B_\lambda)^{Q^{\ast}_{\mu}}(\eta)}{|\eta^{-1}\xi|^{\mu}}{d}\xi{d}\eta.
\end{align*}
On $B(\mathbf{0},\lambda^{-1/2})$ we have $0<B_\lambda\leq B$. Since $q:=Q^{\ast}_{\mu}>1$, convexity gives $(1+t)^q\geq1+qt$ for $t\geq0$, and hence
\[
 (B+B_\lambda)^q\geq B^q+qB^{q-1}B_\lambda
 \qquad\text{on }B(\mathbf{0},\lambda^{-1/2}).
\]
Set
\begin{align*}
 I&:=\int_{B(\mathbf{0},\lambda^{-1/2})}\int_{B(\mathbf{0},\lambda^{-1/2})}
 \frac{B^q(\xi)B^q(\eta)}{|\eta^{-1}\xi|^\mu}{d}\xi{d}\eta,\\
 II&:=\int_{B(\mathbf{0},\lambda^{-1/2})}\int_{B(\mathbf{0},\lambda^{-1/2})}
 \frac{B^q(\xi)B^{q-1}(\eta)B_\lambda(\eta)}{|\eta^{-1}\xi|^\mu}{d}\xi{d}\eta.
\end{align*}
The remaining quadratic cross term is nonnegative, so
\begin{equation}\label{two-bubble-HLS-lower}
 \int_{\mathbb{H}^{n}}\int_{\mathbb{H}^{n}}\frac{|u_\lambda(\xi)|^q|u_\lambda(\eta)|^q}{|\eta^{-1}\xi|^\mu}{d}\xi{d}\eta
 \geq 2I+4qII.
\end{equation}
We next justify the two truncated interaction estimates. Put $R=\lambda^{-1/2}$ and
$p=2Q/(2Q-\mu)$. Since $\|B\|_*=1$, \eqref{ne1}-\eqref{ne2} imply
\[
 0\leq1-I\leq C\int_{B^c(\mathbf0,R)}B^{Q^{\ast}}d\xi
 =O\left(\lambda^{\frac{Q}{2}}\right)
 =o\!\left(\lambda^{\frac{Q-2}{2}}\right).
\]
For $II$, first extend the $\xi$-integration to $\mathbb H^n$. By \eqref{ne1}-\eqref{ne2}, the resulting integral equals
\[
 C(Q,\mu)^{\frac{Q^{\ast}}{2Q^{\ast}_{\mu}}}
 \int_{B(\mathbf0,R)}B^{Q^{\ast}-1}B_\lambda\,d\eta.
\]
The standard one-bubble overlap estimate (see, e.g., \cite[Proposition~3.1]{K25}) gives
\[
 \int_{\mathbb H^n}B^{Q^{\ast}-1}B_\lambda\,d\eta
 =c_0\lambda^{\frac{Q-2}{2}}+o\!\left(\lambda^{\frac{Q-2}{2}}\right),
\]
and the part over $B^c(\mathbf0,R)$ is $O(\lambda^{Q/2})$. It remains to control the error caused by restricting the $\xi$-integration to $B(\mathbf0,R)$. The HLS inequality gives
\begin{align*}
 \int_{B(\mathbf0,R)}\int_{B^c(\mathbf0,R)}
 \frac{B^q(\xi)B^{q-1}(\eta)B_\lambda(\eta)}{|\eta^{-1}\xi|^\mu}
 \,d\xi d\eta\leq C\|B^q\mathbf1_{B^c(\mathbf0,R)}\|_{L^p}
 \|B^{q-1}B_\lambda\mathbf1_{B(\mathbf0,R)}\|_{L^p}
 =o\!\left(\lambda^{\frac{Q-2}{2}}\right),
\end{align*}
where, by the explicit decay of $B$ and the fact that $B_\lambda\leq B$ on $B(\mathbf0,R)$,
\[
 \|B^q\mathbf1_{B^c(\mathbf0,R)}\|_{L^p}
 =O\!\left(\lambda^{\frac{2Q-\mu}{4}}\right)
\]
and
\[
 \|B^{q-1}B_\lambda\mathbf1_{B(\mathbf0,R)}\|_{L^p}
 \leq C\begin{cases}
 \lambda^{\frac{Q-2}{2}},&0<\mu<4,\\
 \lambda^{\frac{Q-2}{2}}|\log\lambda|^{\frac{1}{p}},&\mu=4,\\
 \lambda^{\frac{2Q-\mu}{4}},&4<\mu<Q.
 \end{cases}
\]
In the last case, $(2Q-\mu)/2>(Q-2)/2$ because $\mu<Q<Q+2$. Consequently,
\[
 I=1+o\!\left(\lambda^{\frac{Q-2}{2}}\right),\qquad
 II=C(Q,\mu)^{\frac{Q^{\ast}}{2Q^{\ast}_{\mu}}}c_0\lambda^{\frac{Q-2}{2}}
 +o\!\left(\lambda^{\frac{Q-2}{2}}\right).
\]
Consequently, \eqref{two-bubble-HLS-lower} yields
\[
 \int_{\mathbb{H}^{n}}\int_{\mathbb{H}^{n}}\frac{|u_\lambda(\xi)|^q|u_\lambda(\eta)|^q}{|\eta^{-1}\xi|^\mu}{d}\xi{d}\eta
 \geq 2+4q c_0C(Q,\mu)^{\frac{Q^{\ast}}{2Q^{\ast}_{\mu}}}\lambda^{\frac{Q-2}{2}}
 +o\!\left(\lambda^{\frac{Q-2}{2}}\right).
\]
Taking the $1/q$ power gives
\[
 \left(\int_{\mathbb{H}^{n}}\int_{\mathbb{H}^{n}}\frac{|u_\lambda(\xi)|^q|u_\lambda(\eta)|^q}{|\eta^{-1}\xi|^\mu}{d}\xi{d}\eta\right)^{\frac{1}{q}}
 \geq 2^{\frac{1}{q}}+2^{1+\frac{1}{q}}c_0C(Q,\mu)^{\frac{Q^{\ast}}{2Q^{\ast}_{\mu}}}\lambda^{\frac{Q-2}{2}}
 +o\!\left(\lambda^{\frac{Q-2}{2}}\right),
\]
which proves $(iii)$.}
\end{proof}

{We now prove the two-bubble bound, which is independent of the preceding local perturbation argument.}
\begin{proof}
[Proof of Theorem~\ref{main thm0}-\eqref{upper2}]
{By Proposition~\ref{p4.1}, as $\lambda\to0^+$, we have $u_\lambda\in S^{1,2}(\mathbb H^n)\setminus\mathfrak M$. Put
\[
 a:=\frac{Q-2}{2},\qquad
 A:=c_0C(Q,\mu)^{\frac{Q^{\ast}}{2Q^{\ast}_{\mu}}}>0,
 \qquad q:=Q^{\ast}_{\mu}.
\]
Using parts $(i)$-$(iii)$ of Proposition~\ref{p4.1}, and noting that
$\mathrm{dist}(u_\lambda,\mathfrak M)^2=S_{HL}(Q,\mu)+o(\lambda^a)$, we obtain
\begin{align*}
 \mathcal L(u_\lambda)
 &\leq
 \frac{S_{HL}(Q,\mu)\big[2+2A\lambda^a-2^{1/q}-2^{1/q+1}A\lambda^a\big]
       +o(\lambda^a)}{S_{HL}(Q,\mu)+o(\lambda^a)}\\
 &=2-2^{1/q}-2\big(2^{1/q}-1\big)A\lambda^a+o(\lambda^a).
\end{align*}
Since $q>1$ and $A>0$, the coefficient of $\lambda^a$ is strictly negative. Hence, for all sufficiently small $\lambda>0$,
\[
 H_{NS}\leq\mathcal L(u_\lambda)<2-2^{1/q}
 =2-2^{\frac{1}{Q^{\ast}_{\mu}}}.
\]
This proves \eqref{upper2}.}
\end{proof}

The two strict estimates now play different compactness roles: the spectral estimate separates a minimizing sequence from $\mathfrak{M}$, while the two-bubble estimate prevents an energy split into two nontrivial profiles.

\section{Existence of minimizers: Proof of Theorem~\ref{main thm1}}\label{sec5}
We prove Theorem~\ref{main thm1} by adapting the concentration-compactness scheme of \cite{K25} to the nonlocal Heisenberg quotient. The adaptation uses the nonlocal Br\'{e}zis--Lieb splitting and the strict thresholds established in Section~\ref{sec4}.
Let $\{u_k\}\subset S^{1,2}(\mathbb{H}^{n})\setminus \mathfrak{M}$ be a $\|\cdot\|_*$-normalized minimizing sequence for $H_{NS}$, that is,
 \begin{equation}\label{sequen}
 \mathcal {L}(u_k)=H_{NS}+o_k(1)\quad \mathrm{as}~k\rightarrow+\infty,\quad \|u_k\|_*=1.
 \end{equation}
{By \eqref{eq:l4.1}, we obtain}
\begin{align*}
 \|\nabla_\mathbb{H} u_k\|_{L^2(\mathbb{H}^n)}^2=&(H_{NS}+o_k(1))\mathrm{dist}(u_k,\mathfrak{M})^2+S_{HL}(Q,\mu) \left(\int_{\mathbb{H}^{n}}\int_{\mathbb{H}^{n}}\frac{|u_k(\xi)|^{Q^{\ast}_{\mu}}|u_k(\eta)|
^{Q^{\ast}_{\mu}}}{|\eta^{-1}\xi|^{\mu}}{d}\xi{d}\eta\right)^{\frac{1}{Q^{\ast}_{\mu}}}\\
\leq&
 (H_{NS}+o_k(1))\|\nabla_\mathbb{H} u_k\|_{L^2(\mathbb{H}^n)}^2+S_{HL}(Q,\mu).
\end{align*}
Since $H_{NS}<2-2^{\frac{1}{Q^{\ast}_{\mu}}}<1$ by Theorem~\ref{main thm0}, we know that $\{u_k\}$ is bounded in $S^{1,2}(\mathbb{H}^{n})$. {Moreover, the normalization $\|u_k\|_*=1$ and the sharp HLS inequality imply that $\|u_k\|_{L^{Q^{\ast}}(\mathbb H^n)}$ is bounded away from zero. Hence the concentration--compactness theorem of Lions \cite{L85} excludes vanishing; after applying suitable translations and dilations and passing to a subsequence, we may assume that $u_k\rightharpoonup u_0$ weakly in $S^{1,2}(\mathbb{H}^{n})$ for some nonzero $u_0\in S^{1,2}(\mathbb{H}^{n})$. By local compactness, after a further subsequence we may also assume that $u_k\to u_0$ almost everywhere in $\mathbb H^n$.} {Set $v_k=u_k-u_0$. Then}
\begin{equation}\label{weak}
 {u_k=u_0+v_k,\qquad u_0\in S^{1,2}(\mathbb H^n)\setminus\{0\},\qquad v_k\rightharpoonup0\quad\text{in }S^{1,2}(\mathbb H^n).}
\end{equation}
We first check that if the convergence is strong, then a minimizer of $H_{NS}$ must exist.
\begin{proposition}\label{promini}
Let $Q\geq4$ and $0<\mu<Q$. Assume that either $0<\mu\leq4$ with $\Gamma_{n,\mu}<0$, or \eqref{ach} holds. {If $\{u_k\}$ satisfies \eqref{sequen} and \eqref{weak}, and if $v_k\to0$ strongly in $S^{1,2}(\mathbb{H}^{n})$, then $u_0$ is a minimizer of $H_{NS}$.}
\end{proposition}
\begin{proof}
It is clear that
\begin{equation*}
\|\nabla_\mathbb{H} u_k\|_{L^2(\mathbb{H}^n)}\rightarrow \|\nabla_\mathbb{H} u_0\|_{L^2(\mathbb{H}^n)}\quad \mathrm{and}\quad \mathrm{dist}(u_k,\mathfrak{M})\rightarrow \mathrm{dist}(u_0,\mathfrak{M}),\quad \mathrm{as}~ k\rightarrow+\infty.
\end{equation*}
 {We use the following nonlocal Br\'{e}zis--Lieb lemma; see
\cite[Lemma 2.5]{GS20},}
\begin{equation}\label{BL}
 1=\|u_k\|_*^{2\cdot Q^{\ast}_{\mu}}=\|v_k\|_*^{2\cdot Q^{\ast}_{\mu}}+\|u_0\|_*^{2\cdot Q^{\ast}_{\mu}}+o_k(1),\quad \mathrm{as} ~k\rightarrow+\infty.
\end{equation}
By \eqref{eq:HLSH}, we have $\|v_k\|_*^{2}\leq S_{HL}(Q,\mu)^{-1} \|\nabla_\mathbb{H} v_k\|_{L^2(\mathbb{H}^n)}^2\rightarrow0$.
{Together with \eqref{BL}, this implies}
\begin{equation*}
 \int_{\mathbb{H}^{n}}\int_{\mathbb{H}^{n}}\frac{|u_k(\xi)|^{Q^{\ast}_{\mu}}|u_k(\eta)|
^{Q^{\ast}_{\mu}}}{|\eta^{-1}\xi|^{\mu}}{d}\xi{d}\eta\longrightarrow \int_{\mathbb{H}^{n}}\int_{\mathbb{H}^{n}}\frac{|u_0(\xi)|^{Q^{\ast}_{\mu}}|u_0(\eta)|
^{Q^{\ast}_{\mu}}}{|\eta^{-1}\xi|^{\mu}}{d}\xi{d}\eta,\quad \mathrm{as}~k\rightarrow+\infty.
\end{equation*}
Therefore, if $u_0\notin \mathfrak{M}$, then $\mathcal {L}(u_k)\rightarrow \mathcal {L}(u_0)$, and $u_0$ is a minimizer.

However, the minimizing sequence $\{u_k\}$ cannot satisfy $\mathrm{dist}(u_k,\mathfrak{M})\rightarrow 0$. Otherwise, if $0<\mu\leq4$ and $\Gamma_{n,\mu}<0$, it follows from
\begin{equation*}
\|\nabla_\mathbb{H} u_k\|_{L^2(\mathbb{H}^n)}\geq S_{HL}(Q,\mu)^{\frac{1}{2}} \|u_k\|_*=S_{HL}(Q,\mu)^{\frac{1}{2}}
\end{equation*}
and Lemma \ref{local} that
\begin{equation*}
 H_{NS}= \liminf\limits_{k\rightarrow+\infty}\mathcal{L}(u_k){\geq}H_{NS}^{\mathrm{spec}},
\end{equation*}
which contradicts \eqref{upper1}. If $\mu>4$ and \eqref{ach} holds, then
\begin{equation*}
 H_{NS}<H_{NS}^{2\text{-peak}}= 2-2^{\frac{2n}{4n+4-\mu}} \leq\frac{2}{n+4}.
\end{equation*}
{On the other hand, under the present contradiction assumption $\mathrm{dist}(u_k,\mathfrak M)\to0$, the local expansion for the Folland--Stein--Sobolev deficit in \cite{LZ15} yields
\begin{equation*}
 \liminf_{k\to\infty}
 \frac{\int_{\mathbb{H}^{n}}|\nabla_{\mathbb{H}} u_k|^{2}{d}\xi-S(Q)\big(\int_{\mathbb{H}^{n}}|u_k|^{Q^{\ast}}d\xi\big)^{\frac{2}{Q^{\ast}}}}{\mathrm{dist}(u_k,\mathfrak{M})^2}\geq \frac{2}{n+4}.
\end{equation*}
Here the nonvanishing hypothesis in the local expansion is automatic because $\|u_k\|_*=1$ and \eqref{eq:HLS'} gives $\|u_k\|_{S^{1,2}}\ge S_{HL}(Q,\mu)^{1/2}$. Moreover, by \eqref{eq:HLSH} and \eqref{relation}, the nonlocal deficit is no smaller than the Folland--Stein--Sobolev deficit. Consequently,
\begin{align*}
 H_{NS}=\liminf_{k\to\infty}\mathcal{L}(u_k)
 \geq \liminf_{k\to\infty}
 \frac{\int_{\mathbb{H}^{n}}|\nabla_{\mathbb{H}} u_k|^{2}{d}\xi-S(Q)\big(\int_{\mathbb{H}^{n}}|u_k|^{Q^{\ast}}d\xi\big)^{\frac{2}{Q^{\ast}}}}{\mathrm{dist}(u_k,\mathfrak{M})^2}
 \geq\frac{2}{n+4},
\end{align*}
which is a contradiction.} This proves the proposition.
\end{proof}

\begin{remark}
{\rm From the proof of Proposition~\ref{promini}, we find that the crucial step in proving Proposition~\ref{promini} is establishing that
\begin{equation*}
 H_{NS}<\frac{2}{n+4}\quad \mathrm{when}~\mu>4.
\end{equation*}
Since
\begin{equation*}
 H_{NS}^{\mathrm{spec}}=\frac{(4n+8-2\mu)(\mu+4)}{2(n+4)(4n+8-\mu)}>\frac{2}{n+4} \Longleftrightarrow 0<\mu<Q,
\end{equation*}
we can only hope
\begin{equation*}
 H_{NS}^{2\text{-peak}}=2-2^{\frac{2n}{4n+4-\mu}} \leq\frac{2}{n+4}
\end{equation*}
for some $n\geq1$ and $\mu>4$,
which is equivalent to \eqref{ach}.
}
\end{remark}

{By Proposition~\ref{promini}, it remains only to prove that $v_k\to0$ strongly in $S^{1,2}(\mathbb{H}^{n})$. We therefore analyze the numerator and denominator of $\mathcal{L}(u_k)$ under the decomposition \eqref{weak}.}
{Since $S^{1,2}(\mathbb H^n)$ is a Hilbert space and $v_k\rightharpoonup0$, the cross term with $u_0$ tends to zero; hence}
\begin{equation}\label{BLG}
 \|\nabla_\mathbb{H} u_k\|_{L^2(\mathbb{H}^n)}^2=\|\nabla_\mathbb{H} v_k\|_{L^2(\mathbb{H}^n)}^2+\|\nabla_\mathbb{H} u_0\|_{L^2(\mathbb{H}^n)}^2+o_k(1),\quad \mathrm{as} ~k\rightarrow+\infty.
\end{equation}
{We next decompose $\mathrm{dist}(u_k,\mathfrak M)$ and recall that} $\mathbf{m}(u)=\sup\limits_{v\in \mathfrak{M}_1}\big(u,v^{Q^{\ast}-1}\big)^2$.

\begin{lemma}\label{l5.1}
{Let $\{u_k\}$ satisfy \eqref{weak}. Then}
\begin{equation*}
 \mathbf{m}(u_k)=\max\big\{\mathbf{m}(v_k),\mathbf{m}(u_0)\big\}+o_k(1),\quad \mathrm{as} ~k\rightarrow+\infty.
\end{equation*}
In particular,
\begin{equation*}
 \mathrm{dist}(u_k,\mathfrak{M})^2=\|\nabla_\mathbb{H} v_k\|_{L^2(\mathbb{H}^n)}^2+\|\nabla_\mathbb{H} u_0\|_{L^2(\mathbb{H}^n)}^2-S_{HL}(Q,\mu)C(Q,\mu)^{\frac{Q^{\ast}}{Q^{\ast}_{\mu}}}\max\big
 \{\mathbf{m}(v_k),\mathbf{m}(u_0)\big\}+o_k(1).
\end{equation*}
\end{lemma}
\begin{proof}
{Combining \eqref{eq:l4.1} with \eqref{BLG}, the statement follows directly from} \cite[Lemma 4.2]{K25}.
\end{proof}

{The next lemma is the key consequence of the minimizing property. It balances the best one-bubble correlations of the weak limit and the remainder.}
\begin{lemma}\label{l5.2}
Let $\{u_k\}$ satisfy \eqref{sequen} and \eqref{weak}, and suppose that there exists a constant $\tilde{c}>0$ such that $\|\nabla_\mathbb{H} v_k\|_{L^2(\mathbb{H}^n)}\geq \tilde{c}$. Then
\begin{equation*}
 \mathbf{m}(u_0) =\mathbf{m}(v_k)+o_k(1),\quad \mathrm{as} ~k\rightarrow+\infty.
\end{equation*}
\end{lemma}
\begin{proof}
Assume for contradiction that, up to a subsequence,
\begin{equation}\label{contra}
\lim\limits_{k\rightarrow+\infty} \mathbf{m}(v_k)> \mathbf{m}(u_0).
\end{equation}
Since $u_0\not \equiv0$, we have $\mathbf{m}(u_0)>0$ by Lemma \ref{l4.1}.
Using \eqref{e4.7} and \eqref{contra}, we obtain
\begin{equation*}
 0<\mathbf{m}(u_0)< \lim\limits_{k\rightarrow+\infty}\mathbf{m}(v_k)\leq C(Q,\mu)^{-\frac{Q^{\ast}}{Q^{\ast}_{\mu}}}\lim\limits_{k\rightarrow+\infty}\|v_k\|_*^2,
\end{equation*}
Together with \eqref{BL}, this implies that $\|v_k\|_*$ is bounded away from zero and infinity. Multiplying $u_k$ by $\frac{1}{\|v_k\|_*}$, we consider
\begin{equation*}
 \tilde{u}_k=\frac{v_k}{\|v_k\|_*}+\frac{u_0}{\|v_k\|_*}=:\tilde{v}_k+\tilde{u}_{0,k},\quad \mathrm{with}~\|\tilde{v}_k\|_*=1.
\end{equation*}
Then, by \eqref{BL}-\eqref{contra} and Lemma \ref{l5.1}, we have
\begin{align*}
 \mathcal{L}(\tilde{u}_k)=&\frac{\|\nabla_\mathbb{H} \tilde{u}_k\|_{L^2(\mathbb{H}^n)}^2-S_{HL}(Q,\mu) \Big(\int_{\mathbb{H}^{n}}\int_{\mathbb{H}^{n}}\frac{|\tilde{u}_k(\xi)|^{Q^{\ast}_{\mu}}|\tilde{u}_k(\eta)|
^{Q^{\ast}_{\mu}}}{|\eta^{-1}\xi|^{\mu}}{d}\xi{d}\eta\Big)^{\frac{1}{Q^{\ast}_{\mu}}}}{\mathrm{dist}(\tilde{u}_k,\mathfrak{M})^2}\\
=&\frac{\|\nabla_\mathbb{H} \tilde{v}_k\|_{L^2(\mathbb{H}^n)}^2+\|\nabla_\mathbb{H} \tilde{u}_{0,k}\|_{L^2(\mathbb{H}^n)}^2-S_{HL}(Q,\mu) \Big(1+\|\tilde{u}_{0,k}\|_*^{2\cdot Q^{\ast}_{\mu}}\Big)^{\frac{1}{Q^{\ast}_{\mu}}}}
{\|\nabla_\mathbb{H} \tilde{v}_k\|_{L^2(\mathbb{H}^n)}^2+\|\nabla_\mathbb{H} \tilde{u}_{0,k}\|_{L^2(\mathbb{H}^n)}^2-S_{HL}(Q,\mu)C(Q,\mu)^{\frac{Q^{\ast}}{Q^{\ast}_{\mu}}}\mathbf{m}(\tilde{v}_k)}+o_k(1)\\
=&\frac{\|\nabla_\mathbb{H} \tilde{v}_k\|_{L^2(\mathbb{H}^n)}^2-S_{HL}(Q,\mu)+\|\nabla_\mathbb{H} \tilde{u}_{0,k}\|_{L^2(\mathbb{H}^n)}^2-S_{HL}(Q,\mu) \left[\Big(1+\|\tilde{u}_{0,k}\|_*^{2\cdot Q^{\ast}_{\mu}}\Big)^{\frac{1}{Q^{\ast}_{\mu}}}-1\right]}
{\|\nabla_\mathbb{H} \tilde{v}_k\|_{L^2(\mathbb{H}^n)}^2-S_{HL}(Q,\mu)C(Q,\mu)^{\frac{Q^{\ast}}{Q^{\ast}_{\mu}}}\mathbf{m}(\tilde{v}_k)+
\|\nabla_\mathbb{H} \tilde{u}_{0,k}\|_{L^2(\mathbb{H}^n)}^2}+o_k(1).
\end{align*}

{We next estimate the quotient above using \cite[Lemmas 2.3-2.4]{K25}.} {Passing to a further subsequence if necessary, we may assume that all scalar sequences entering the limits below converge.} Suppose for the moment that $\tilde{v}_k\notin \mathfrak{M}$ for all $k$. Then set
 \begin{equation*}
 A:=\lim\limits_{k\rightarrow+\infty}\|\nabla_\mathbb{H} \tilde{v}_k\|_{L^2(\mathbb{H}^n)}^2-S_{HL}(Q,\mu),\quad B:=\lim\limits_{k\rightarrow+\infty}\left(\|\nabla_\mathbb{H} \tilde{v}_k\|_{L^2(\mathbb{H}^n)}^2-S_{HL}(Q,\mu)C(Q,\mu)^{\frac{Q^{\ast}}{Q^{\ast}_{\mu}}}\mathbf{m}(\tilde{v}_k)\right)
 \end{equation*}
 \begin{equation*}
 C:=\lim\limits_{k\rightarrow+\infty}\left\{\|\nabla_\mathbb{H} \tilde{u}_{0,k}\|_{L^2(\mathbb{H}^n)}^2-S_{HL}(Q,\mu) \left[\Big(1+\|\tilde{u}_{0,k}\|_*^{2\cdot Q^{\ast}_{\mu}}\Big)^{\frac{1}{Q^{\ast}_{\mu}}}-1\right]\right\},\quad D:=\lim\limits_{k\rightarrow+\infty}\|\nabla_\mathbb{H} \tilde{u}_{0,k}\|_{L^2(\mathbb{H}^n)}^2.
 \end{equation*}
 Notice that $A,B,C,D>0$ because $\tilde{v}_k\notin \mathfrak{M}$, $\|\nabla_\mathbb{H} \tilde{u}_{0,k}\|_{L^2(\mathbb{H}^n)}$ is bounded away from zero, and $(1+x)^\alpha<1+x^\alpha$ for any $x>0$ and $0<\alpha<1$. Since
 \begin{equation*}
 H_{NS}=\lim\limits_{k\rightarrow+\infty} \mathcal{L}(u_k)=\lim\limits_{k\rightarrow+\infty} \mathcal{L} (\tilde{u}_k)=\frac{A+C}{B+D},
 \end{equation*}
 and
 \begin{equation*}
 H_{NS}\leq \lim\limits_{k\rightarrow+\infty} \mathcal{L}(v_k)=\lim\limits_{k\rightarrow+\infty} \mathcal{L} (\tilde{v}_k)=\frac{A}{B},
 \end{equation*}
 {hence $\frac{A}{B}\geq H_{NS}\geq\frac{C}{D}$.} {Choose $c_k>0$ and set $F_k=c_k\tilde{u}_{0,k}$ so that $\mathbf{m}(F_k)=\mathbf{m}(\tilde{v}_k)$. By \eqref{contra},}
 \begin{equation*}
 \lim\limits_{k\rightarrow+\infty}c_k^2=\lim\limits_{k\rightarrow+\infty}\frac{\mathbf{m}(\tilde{v}_k)}{\mathbf{m}(\tilde{u}_{0,k})}=
 \frac{\lim\limits_{k\rightarrow+\infty}\mathbf{m}(v_k)}{\mathbf{m}(u_0)}>1.
 \end{equation*}
 Thus,
 \begin{equation*}
 {\lim\limits_{k\rightarrow+\infty}\|F_k\|_*=\lim\limits_{k\rightarrow+\infty}c_k\|\tilde{u}_{0,k}\|_*>\lim\limits_{k\rightarrow+\infty}\|\tilde{u}_{0,k}\|_*.}
 \end{equation*}
 {By \cite[Lemma 2.3]{K25}, for every $r>2$,
 \begin{equation}\label{K253}
 \eta\longmapsto \frac{(1+\eta^r)^{\frac{2}{r}}-1}{\eta^2}\quad\text{is strictly increasing on }(0,\infty).
 \end{equation}
Take $r=2Q^{\ast}_{\mu}>2$ so that $(1+\eta^r)^{2/r}=(1+\eta^{2Q^{\ast}_{\mu}})^{1/Q^{\ast}_{\mu}}$. Thus}
 \begin{align*}
 \frac{C}{D}=&1-\lim\limits_{k\rightarrow+\infty}\frac{S_{HL}(Q,\mu) \left[\Big(1+\|\tilde{u}_{0,k}\|_*^{2\cdot Q^{\ast}_{\mu}}\Big)^{\frac{1}{Q^{\ast}_{\mu}}}-1\right]}{\|\nabla_\mathbb{H} \tilde{u}_{0,k}\|_{L^2(\mathbb{H}^n)}^2}\\
 =&1-\lim\limits_{k\rightarrow+\infty}\frac{S_{HL}(Q,\mu) \left[\Big(1+\|\tilde{u}_{0,k}\|_*^{2\cdot Q^{\ast}_{\mu}}\Big)^{\frac{1}{Q^{\ast}_{\mu}}}-1\right]}{S_{[\tilde{u}_{0,k}]}\|\tilde{u}_{0,k}\|_*^2}\\
 >&1-\lim\limits_{k\rightarrow+\infty}\frac{S_{HL}(Q,\mu) \left[\Big(1+\|F_k\|_*^{2\cdot Q^{\ast}_{\mu}}\Big)^{\frac{1}{Q^{\ast}_{\mu}}}-1\right]}{S_{[\tilde{u}_{0,k}]}\|F_k\|_*^2}:=\frac{E}{F},
 \end{align*}
 where
 \begin{equation*}
 S_{[\tilde{u}_{0,k}]}=\frac{\|\nabla_\mathbb{H} \tilde{u}_{0,k}\|_{L^2(\mathbb{H}^n)}^2}{\|\tilde{u}_{0,k}\|_*^2}\geq S_{HL}(Q,\mu)>0,
 \end{equation*}
 and
 \begin{equation*}
 E:=\lim\limits_{k\rightarrow+\infty}\left\{S_{[\tilde{u}_{0,k}]}\|F_k\|_*^2-S_{HL}(Q,\mu) \left[\Big(1+\|F_k\|_*^{2\cdot Q^{\ast}_{\mu}}\Big)^{\frac{1}{Q^{\ast}_{\mu}}}-1\right]\right\},\quad F:=\lim\limits_{k\rightarrow+\infty}S_{[\tilde{u}_{0,k}]}\|F_k\|_*^2.
 \end{equation*}
 {By Lemma~\ref{l5.1}, $\mathbf{m}(F_k+\tilde{v}_k)=\mathbf{m}(\tilde{v}_k)+o_k(1)$ for $k$ large enough. Hence}
 \begin{align*}
 \mathrm{dist}(F_k+\tilde{v}_k,\mathfrak{M})^2=&\|\nabla_\mathbb{H} (F_k+\tilde{v}_k)\|_{L^2(\mathbb{H}^n)}^2-S_{HL}(Q,\mu)C(Q,\mu)^{\frac{Q^{\ast}}{Q^{\ast}_{\mu}}}
 \mathbf{m}(F_k+\tilde{v}_k)+o_k(1)\\
 =&\|\nabla_\mathbb{H} F_k\|_{L^2(\mathbb{H}^n)}^2+\|\nabla_\mathbb{H}\tilde{v}_k\|_{L^2(\mathbb{H}^n)}^2-S_{HL}(Q,\mu)C(Q,\mu)^{\frac{Q^{\ast}}{Q^{\ast}_{\mu}}}
 \mathbf{m}(\tilde{v}_k)+o_k(1)\\
 \geq &\|\nabla_\mathbb{H}\tilde{v}_k\|_{L^2(\mathbb{H}^n)}^2-S_{HL}(Q,\mu)C(Q,\mu)^{\frac{Q^{\ast}}{Q^{\ast}_{\mu}}}
 \mathbf{m}(\tilde{v}_k)+o_k(1)\\
 =&\mathrm{dist}(\tilde{v}_k,\mathfrak{M})^2>0,
 \end{align*}
{which implies $F_k+\tilde v_k\notin\mathfrak M$. Moreover,}
 \begin{align*}
 &\lim\limits_{k\rightarrow+\infty} \mathcal{L}(F_k+\tilde{v}_k)\\=&
\lim\limits_{k\rightarrow+\infty}\frac{\|\nabla_\mathbb{H} \tilde{v}_k\|_{L^2(\mathbb{H}^n)}^2-S_{HL}(Q,\mu)+\|\nabla_\mathbb{H} F_k\|_{L^2(\mathbb{H}^n)}^2-S_{HL}(Q,\mu) \left[\Big(1+\|F_k\|_*^{2\cdot Q^{\ast}_{\mu}}\Big)^{\frac{1}{Q^{\ast}_{\mu}}}-1\right]}
{\|\nabla_\mathbb{H} \tilde{v}_k\|_{L^2(\mathbb{H}^n)}^2-S_{HL}(Q,\mu)C(Q,\mu)^{\frac{Q^{\ast}}{Q^{\ast}_{\mu}}}\mathbf{m}(\tilde{v}_k)+
\|\nabla_\mathbb{H} F_k\|_{L^2(\mathbb{H}^n)}^2}\\
=&\lim\limits_{k\rightarrow+\infty}\frac{\|\nabla_\mathbb{H} \tilde{v}_k\|_{L^2(\mathbb{H}^n)}^2-S_{HL}(Q,\mu)+S_{[\tilde{u}_{0,k}]}\|F_k\|_*^2-S_{HL}(Q,\mu) \left[\Big(1+\|F_k\|_*^{2\cdot Q^{\ast}_{\mu}}\Big)^{\frac{1}{Q^{\ast}_{\mu}}}-1\right]}
{\|\nabla_\mathbb{H} \tilde{v}_k\|_{L^2(\mathbb{H}^n)}^2-S_{HL}(Q,\mu)C(Q,\mu)^{\frac{Q^{\ast}}{Q^{\ast}_{\mu}}}\mathbf{m}(\tilde{v}_k)+
S_{[\tilde{u}_{0,k}]}\|F_k\|_*^2}
=\frac{A+E}{B+F},
\end{align*}
{where we used}
\begin{equation*}
 S_{[\tilde{u}_{0,k}]}\|F_k\|_*^2=\frac{\|\nabla_\mathbb{H} \tilde{u}_{0,k}\|_{L^2(\mathbb{H}^n)}^2}{\|\tilde{u}_{0,k}\|_*^2}c_k^2\|\tilde{u}_{0,k}\|_*^2
 =c_k^2\|\nabla_\mathbb{H} \tilde{u}_{0,k}\|_{L^2(\mathbb{H}^n)}^2=\|\nabla_\mathbb{H} F_k\|_{L^2(\mathbb{H}^n)}^2.
\end{equation*}
Since
\begin{equation*}
 D=\lim\limits_{k\rightarrow+\infty}\|\nabla_\mathbb{H} \tilde{u}_{0,k}\|_{L^2(\mathbb{H}^n)}^2<\lim\limits_{k\rightarrow+\infty}c_k^2\|\nabla_\mathbb{H} \tilde{u}_{0,k}\|_{L^2(\mathbb{H}^n)}^2
 =\lim\limits_{k\rightarrow+\infty}S_{[\tilde{u}_{0,k}]}\|F_k\|_*^2=F,
\end{equation*}
 and
 \begin{equation*}
 \frac{A}{B}\geq \frac{C}{D}>\frac{E}{F},
 \end{equation*}
 using the algebraic inequality in \cite[Lemma 2.4]{K25}, we obtain
 \begin{equation*}
 H_{NS}=\lim\limits_{k\rightarrow+\infty} \mathcal{L}(u_k)=\lim\limits_{k\rightarrow+\infty} \mathcal{L} (\tilde{u}_k)=\frac{A+C}{B+D}>\frac{A+E}{B+F}=\lim\limits_{k\rightarrow+\infty} \mathcal{L}(F_k+\tilde{v}_k),
 \end{equation*}
{which contradicts $F_k+\tilde v_k\notin\mathfrak M$. Hence \eqref{contra} is impossible.}

{If $\tilde v_k\in\mathfrak M$ along a subsequence, then $A=B=0$, and the same argument gives a contradiction from $\frac{C}{D}>\frac{E}{F}$.}

 The remaining case to treat is that, up to a subsequence,
\begin{equation*}
 \mathbf{m}(u_0)>\lim\limits_{k\rightarrow+\infty} \mathbf{m}(v_k).
\end{equation*}
But here one arrives at a contradiction in a similar fashion, with the roles of $v_k$ and $u_0$
reversed, and considering
\begin{equation*}
 \hat{u}_k=\frac{v_k}{\|u_0\|_*}+\frac{u_0}{\|u_0\|_*}=:\hat{v}_k+\hat{u}_{0},\quad \mathrm{with}~\|\hat{u}_0\|_*=1.
\end{equation*}
Then
\begin{equation*}
\hat{D}:=\lim\limits_{k\rightarrow+\infty}\|\nabla_\mathbb{H} \hat{v}_{k}\|_{L^2(\mathbb{H}^n)}^2=\lim\limits_{k\rightarrow+\infty}\frac{\|\nabla_\mathbb{H} {v}_{k}\|_{L^2(\mathbb{H}^n)}^2}{\|u_0\|_*^2}\geq \frac{\tilde{c}^2}{\|u_0\|_*^2}>0.
\end{equation*}
The rest of the proof is identical to the above.
\end{proof}

We are now ready to prove the existence of minimizers for $H_{NS}$.
\begin{proof}
[Proof of Theorem~\ref{main thm1}]
{Let $\{u_k\}$ be a minimizing sequence satisfying \eqref{sequen} and \eqref{weak}. Assume for contradiction that} $v_k=u_k-u_0$ does not converge strongly to zero in $S^{1,2}(\mathbb{H}^{n})$. Then, up to a subsequence, we have $\|\nabla_\mathbb{H} v_k\|_{L^2(\mathbb{H}^n)}\geq c$ for some $c>0$. Thus, by Lemma \ref{l5.2}, we have
\begin{equation}\label{toprove}
 \mathbf{m}(u_0) =\mathbf{m}(v_k)+o_k(1),\quad \mathrm{as} ~k\rightarrow+\infty.
\end{equation}

Suppose first that $\|u_0\|_*\geq\|v_k\|_*+o_k(1)$. As in the proof of Lemma \ref{l5.2}, $\|v_k\|_*$ is bounded away from zero and infinity.
Without loss of generality,
we may assume that $\|u_0\|_*=1$ by multiplying by a suitable scalar factor. Using
\eqref{BL}, \eqref{BLG}, \eqref{toprove} and Lemma \ref{l5.1}, we obtain
\begin{align*}
 H_{NS}=&\lim\limits_{k\rightarrow+\infty}\mathcal{L}(u_k)\\
 =&
\lim\limits_{k\rightarrow+\infty}\frac{\|\nabla_\mathbb{H} u_0\|_{L^2(\mathbb{H}^n)}^2-S_{HL}(Q,\mu)+\|\nabla_\mathbb{H} v_k\|_{L^2(\mathbb{H}^n)}^2-S_{HL}(Q,\mu) \left[\Big(1+\|v_k\|_*^{2\cdot Q^{\ast}_{\mu}}\Big)^{\frac{1}{Q^{\ast}_{\mu}}}-1\right]}
{\|\nabla_\mathbb{H} u_0\|_{L^2(\mathbb{H}^n)}^2-S_{HL}(Q,\mu)C(Q,\mu)^{\frac{Q^{\ast}}{Q^{\ast}_{\mu}}}\mathbf{m}(u_0)+
\|\nabla_\mathbb{H} v_k\|_{L^2(\mathbb{H}^n)}^2}.
\end{align*}
{As in the proof of Lemma~\ref{l5.2}, \eqref{K253} gives}
\begin{align*}
 H_{NS}\geq& \lim\limits_{k\rightarrow+\infty}\frac{\|\nabla_\mathbb{H} v_k\|_{L^2(\mathbb{H}^n)}^2-S_{HL}(Q,\mu) \left[\Big(1+\|v_k\|_*^{2\cdot Q^{\ast}_{\mu}}\Big)^{\frac{1}{Q^{\ast}_{\mu}}}-1\right]}
{
\|\nabla_\mathbb{H} v_k\|_{L^2(\mathbb{H}^n)}^2}\\
=&1-\lim\limits_{k\rightarrow+\infty}\frac{S_{HL}(Q,\mu) \left[\Big(1+\|v_k\|_*^{2\cdot Q^{\ast}_{\mu}}\Big)^{\frac{1}{Q^{\ast}_{\mu}}}-1\right]}
{
S_{[v_{k}]}\|v_{k}\|_*^2}\\
\geq& 1-\frac{S_{HL}(Q,\mu) \Big(2^{\frac{1}{Q^{\ast}_{\mu}}}-1\Big)}
{
S_{[v_{k}]}}\geq 2-2^{\frac{1}{Q^{\ast}_{\mu}}},
\end{align*}
since
 \begin{equation*}
 S_{[v_{k}]}=\frac{\|\nabla_\mathbb{H} v_{k}\|_{L^2(\mathbb{H}^n)}^2}{\|v_{k}\|_*^2}\geq S_{HL}(Q,\mu)>0.
 \end{equation*}
However, this is impossible, because $H_{NS}<2-2^{\frac{1}{Q^{\ast}_{\mu}}}$ by Theorem~\ref{main thm0}.

If we assume instead the reverse inequality $\|u_0\|_*\leq\|v_k\|_*+o_k(1)$, we can obtain a contradiction by arguing in exactly the same way with the roles of $v_k$ and $u_0$ reversed.
{Thus $v_k\to0$ strongly in $S^{1,2}(\mathbb H^n)$. Proposition~\ref{promini} then completes the proof of Theorem~\ref{main thm1}.}
\end{proof}

Having settled the lower stability constant and its attainability, we turn to the reverse comparison. This question is independent of compactness of minimizing sequences and is governed by high spherical modes.

\section{\texorpdfstring{{Sharp universal upper bound: Proof of Theorem~\ref{main thm2}}}{Sharp universal upper bound: Proof of Theorem 3.3}}\label{sec6}

In this section, we establish the sharp universal upper bound $H_{UB}=1$ and show that equality cannot occur outside the extremal manifold.

\begin{proof}
[Proof of Theorem~\ref{main thm2}]
{By \eqref{eq:l4.1} and \eqref{e4.7}, for every $u\in S^{1,2}(\mathbb H^n)$, we have}
\begin{equation*}
 \mathrm{dist}(u,\mathfrak{M})^2=\|\nabla_\mathbb{H} u\|_{L^2(\mathbb{H}^n)}^2-S_{HL}(Q,\mu)C(Q,\mu)^{\frac{Q^{\ast}}{Q^{\ast}_{\mu}}}\mathbf{m}(u)\geq
 \|\nabla_\mathbb{H} u\|_{L^2(\mathbb{H}^n)}^2-S_{HL}(Q,\mu) \|u\|_*^2.
\end{equation*}
{Let $\bar c$ denote the optimal constant in the reverse inequality. The preceding estimate shows that $0<\bar c\leq1$ and that}
\begin{equation}\label{uppmini}
 \int_{\mathbb{H}^{n}}|\nabla_{\mathbb{H}} u|^{2}{d}\xi-S_{HL}(Q,\mu) \left(\int_{\mathbb{H}^{n}}\int_{\mathbb{H}^{n}}\frac{|u(\xi)|^{Q^{\ast}_{\mu}}|u(\eta)|
^{Q^{\ast}_{\mu}}}{|\eta^{-1}\xi|^{\mu}}{d}\xi{d}\eta\right)^{\frac{1}{Q^{\ast}_{\mu}}}\leq \bar{c}\ \mathrm{dist}(u,\mathfrak{M})^2,\ \ \forall \, u\in S^{1,2}(\mathbb{H}^{n}).
\end{equation}

{To prove sharpness, fix $m\geq1$ and choose a nonzero real-valued spherical harmonic
\[
 \omega_m\in(\mathcal H_{m,m}^{n+1})_{\mathbb R}.
\]
Set
\begin{equation*}
 \phi_m(\xi)=[\mathcal J_{\mathcal C}(\xi)]^{\frac1{Q^\ast}}\omega_m(\mathcal C\xi)
 =(\mathcal C^*\omega_m)(\xi),
\end{equation*}
and consider the perturbations
\begin{equation*}
 f_k=U_\mu+\epsilon_k\phi_m,
\end{equation*}
where $U_\mu=\mathcal{A}U$ with
\[
 \mathcal A=S(Q)^{\frac{(Q-\mu)(2-Q)}{4(Q+2-\mu)}}
 C(Q,\mu)^{\frac{2-Q}{2(Q+2-\mu)}},
 \qquad \epsilon_k\to0^+\quad\text{as }k\to+\infty.
\]
Since $(m,m)$ is disjoint from the tangent modes $(0,0),(1,0),(0,1)$, the real-block decomposition gives
\[
 \phi_m\perp T_{U_\mu}\mathfrak M
 \quad\text{in }S^{1,2}(\mathbb H^n).
\]
Furthermore, $\phi_m$ is a common eigenfunction of the two spectral problems and satisfies
\begin{equation*}
 -\Delta_{\mathbb{H}} \phi_m=\nu_{m,m}\left(\int_{\mathbb{H}^{n}}\frac{|U_\mu(\eta)|
^{Q^{\ast}_{\mu}}}{|\eta^{-1}\xi|^{\mu}}{d}\eta\right)|U_\mu|^{Q^{\ast}_{\mu}-2}\phi_m,\qquad \xi\in\mathbb{H}^{n},
\end{equation*}
and
\begin{equation*}
 -\Delta_{\mathbb{H}} \phi_m=\tau_{m,m}\left(\int_{\mathbb{H}^{n}}\frac{|U_\mu(\eta)|
^{Q^{\ast}_{\mu}-1}\phi_m(\eta)}{|\eta^{-1}\xi|^{\mu}}{d}\eta\right)|U_\mu|^{Q^{\ast}_{\mu}-1},\qquad \xi\in\mathbb{H}^{n}.
\end{equation*}
For each fixed $m$, the same Taylor expansion used in the proof of Theorem~\ref{main thm0}-\eqref{upper1} in Section~\ref{sec41} gives
\begin{equation*}
 \|f_k\|^2_*=\|U_\mu\|_*^2+\epsilon_k^2 \|U_\mu\|_*^{2(1-Q^{\ast}_{\mu})}
\left(\frac{Q^{\ast}_{\mu}-1}{\nu_{m,m}}+\frac{Q^{\ast}_{\mu}}{\tau_{m,m}}\right)\|\nabla_\mathbb{H}\phi_m\|^2_{{L}^{2}(\mathbb{H}^n)}+
o\big(\epsilon_k^2\big).
\end{equation*}
Moreover, since $\phi_m$ is normal to $T_{U_\mu}\mathfrak M$, the tubular-neighborhood projection gives
\[
 \mathrm{dist}(f_k,\mathfrak{M})^2=\epsilon_k^2
 \|\nabla_\mathbb{H}\phi_m\|^2_{{L}^{2}(\mathbb{H}^n)}
\]
for $k$ large enough. Therefore,
 \begin{align*}
 &\bar c\geq\frac{\int_{\mathbb{H}^{n}}|\nabla_{\mathbb{H}} f_k|^{2}{d}\xi-S_{HL}(Q,\mu) \Big(\int_{\mathbb{H}^{n}}\int_{\mathbb{H}^{n}}\frac{|f_k(\xi)|^{Q^{\ast}_{\mu}}|f_k(\eta)|
^{Q^{\ast}_{\mu}}}{|\eta^{-1}\xi|^{\mu}}{d}\xi{d}\eta\Big)^{\frac{1}{Q^{\ast}_{\mu}}}}{\mathrm{dist}(f_k,\mathfrak{M})^2} \\
=&\frac{1}{\epsilon_k^2\|\nabla_\mathbb{H}\phi_m\|^2_{{L}^{2}(\mathbb{H}^n)}}\Bigg\{
\|\nabla_\mathbb{H}U_\mu\|^2_{{L}^{2}(\mathbb{H}^n)}
+\epsilon_k^2\|\nabla_\mathbb{H}\phi_m\|^2_{{L}^{2}(\mathbb{H}^n)}\\
&\qquad-S_{HL}(Q,\mu)\Bigg[\|U_\mu\|_*^2
+\epsilon_k^2\|\nabla_\mathbb{H}\phi_m\|^2_{{L}^{2}(\mathbb{H}^n)}
\|U_\mu\|_*^{2(1-Q^{\ast}_{\mu})}
\left(\frac{Q^{\ast}_{\mu}-1}{\nu_{m,m}}+\frac{Q^{\ast}_{\mu}}{\tau_{m,m}}\right)\Bigg]\Bigg\}+\frac{o\big(\epsilon_k^2\big)}{\epsilon_k^2}\\
=&1-\left(\frac{Q^{\ast}_{\mu}-1}{\nu_{m,m}}+\frac{Q^{\ast}_{\mu}}{\tau_{m,m}}\right)+o_k(1)\\
&\longrightarrow 1-\left(\frac{Q^{\ast}_{\mu}-1}{\nu_{m,m}}+\frac{Q^{\ast}_{\mu}}{\tau_{m,m}}\right)
\qquad\text{as }k\to+\infty.
 \end{align*}
Here we used
 \begin{equation*}
 S_{HL}(Q,\mu) \|U_\mu\|_*^2=\|\nabla_\mathbb{H}U_\mu\|^2_{{L}^{2}(\mathbb{H}^n)}=\|U_\mu\|_*^{2\cdot Q^{\ast}_{\mu}}.
 \end{equation*}
It follows that
 \begin{equation*}
 \bar{c}\geq 1-\left(\frac{Q^{\ast}_{\mu}-1}{\nu_{m,m}}+\frac{Q^{\ast}_{\mu}}{\tau_{m,m}}\right)
 \end{equation*}
for every $m\ge1$. Since $\nu_{m,m},\tau_{m,m}\to+\infty$ as $m\to\infty$, letting $m\to\infty$ yields $\bar c\geq1$. Since already $\bar c\leq1$, we conclude that $\bar c=1$.}

 {Finally, equality in \eqref{uppmini} cannot occur outside $\mathfrak{M}$. Otherwise,} there exists $\bar{u}\in S^{1,2}(\mathbb{H}^{n})\setminus \mathfrak{M}$ such that
 \begin{equation*}
\mathrm{dist}(\bar{u},\mathfrak{M})^2= \int_{\mathbb{H}^{n}}|\nabla_{\mathbb{H}} \bar{u}|^{2}{d}\xi-S_{HL}(Q,\mu) \left(\int_{\mathbb{H}^{n}}\int_{\mathbb{H}^{n}}\frac{|\bar{u}(\xi)|^{Q^{\ast}_{\mu}}|\bar{u}(\eta)|
^{Q^{\ast}_{\mu}}}{|\eta^{-1}\xi|^{\mu}}{d}\xi{d}\eta\right)^{\frac{1}{Q^{\ast}_{\mu}}},
\end{equation*}
which is equivalent to $\mathbf{m}(\bar{u})= C(Q,\mu)^{-\frac{Q^{\ast}}{Q^{\ast}_{\mu}}}\|\bar{u}\|_*^2$. However, this is true if and only if $\bar{u}\in \mathfrak{M}$ by \eqref{e4.7}, a contradiction. The proof of Theorem~\ref{main thm2} is thereby completed.
\end{proof}

The functional deficit is only one measure of stability. We finally consider the Euler--Lagrange residual, for which the numerator is linearized through the inverse sub-Laplacian and therefore requires a separate expansion.

\section{\texorpdfstring{{Critical-point quantitative stability: Proof of Theorem~\ref{addthm}}}{Critical-point quantitative stability: Proof of Theorem 3.4}}\label{sec7}

\begingroup
We use the same spherical normalization as in Subsection~\ref{sec41}. Put
\[
 q:=Q^{\ast}_{\mu}=\frac{4n+4-\mu}{2n},
\]
and let $\mathscr A$ and $\mathcal B_\mu$ be the normalized CR energy operator and the normalized HLS bilinear form introduced there. For $v>0$, define the coefficient-one Euler--Lagrange residual weakly by
\begin{equation}\label{sphere-residual}
 \big\langle\mathscr R_\mu(v),h\big\rangle
 :=\langle\mathscr Av,h\rangle-\mathcal B_\mu(v^q,v^{q-1}h)
\end{equation}
for every test function $h$. Its dual energy norm is
\[
 \|F\|_{\mathscr A^{-1}}^2:=\langle F,\mathscr A^{-1}F\rangle.
\]
Let $\mathfrak M_{1,\mathbb S}$ denote the image, under this normalization, of the fixed-amplitude single-bubble family $\mathfrak M_1$. The constant function $1$ belongs to $\mathfrak M_{1,\mathbb S}$ and satisfies $\mathscr R_\mu(1)=0$.

\begin{proof}
[Proof of Theorem~\ref{addthm}]
Set
\[
 \nu_2:=\nu_{2,0}=\frac{n+4}{n},
 \qquad
 \vartheta_2:=\vartheta_{2,0}
 =\frac{\mu(\mu+4)}{(4n+4-\mu)(4n+8-\mu)},
\]
and
\begin{equation}\label{CP-ell2}
 \ell_2:=\nu_2-(q-1)-q\vartheta_2.
\end{equation}
A direct simplification gives
\begin{equation}\label{CP-kappa}
 \frac{\ell_2}{\nu_2}
 =\frac{(4n+8-2\mu)(\mu+4)}{2(n+4)(4n+8-\mu)}
 =H_{NS}^{\rm spec}.
\end{equation}
Moreover,
\[
 \ell_2=\frac{(\mu+4)(2n+4-\mu)}{n(4n+8-\mu)}>0
\]
because $0<\mu<Q=2n+2$.

Let
\[
 v_\epsilon=1+\epsilon\phi,
 \qquad
 \phi(\zeta)=\operatorname{Re}(\zeta_{n+1}^2)
\]
with $\phi$ as in \eqref{phiLi}. Recall that
\begin{equation}\label{CP-M2}
 M_2:=\int_{\mathbb S^{2n+1}}\phi^2\,d\sigma
 =\frac{1}{(n+1)(n+2)},
 \qquad
 \mathscr A\phi=\nu_2\phi.
\end{equation}
The fixed-amplitude orbit has tangent space
\[
 T_1\mathfrak M_{1,\mathbb S}
 =\big(\mathcal H_{1,0}^{n+1}\oplus\mathcal H_{0,1}^{n+1}\big)_{\mathbb R}.
\]
Since $\phi\in(\mathcal H_{2,0}^{n+1}\oplus\mathcal H_{0,2}^{n+1})_{\mathbb R}$ and $\mathscr A$ is diagonal on the spherical-harmonic decomposition, $\phi$ is $\mathscr A$-orthogonal to $T_1\mathfrak M_{1,\mathbb S}$. The standard nearest-point expansion for a smooth finite-dimensional submanifold of a Hilbert space therefore yields
\begin{equation}\label{CP-distance}
 \operatorname{dist}_{\mathscr A}(v_\epsilon,\mathfrak M_{1,\mathbb S})^2
 =\epsilon^2\langle\phi,\mathscr A\phi\rangle+o(\epsilon^2)
 =\nu_2M_2\epsilon^2+o(\epsilon^2).
\end{equation}
In particular, $v_\epsilon\notin\mathfrak M_{1,\mathbb S}$ for all sufficiently small $\epsilon\ne0$.

We next linearize the residual. Since $\phi$ is smooth and bounded, $v_\epsilon$ stays uniformly away from zero for $|\epsilon|$ small. Hence, uniformly on $\mathbb S^{2n+1}$,
\[
 (1+\epsilon\phi)^q=1+q\epsilon\phi+O(\epsilon^2),
 \qquad
 (1+\epsilon\phi)^{q-1}=1+(q-1)\epsilon\phi+O(\epsilon^2).
\]
Using \eqref{sphere-residual}, $\mathscr R_\mu(1)=0$, and the symmetry of $\mathcal B_\mu$, we obtain
\begin{equation}\label{CP-linearization}
 \mathscr R_\mu(v_\epsilon)=\epsilon L\phi+O_{\mathscr A^{-1}}(\epsilon^2),
\end{equation}
where
\[
 \langle Lr,h\rangle
 :=\langle\mathscr Ar,h\rangle
 -(q-1)\int_{\mathbb S^{2n+1}}rh\,d\sigma
 -q\mathcal B_\mu(r,h).
\]
Here the remainder estimate in \eqref{CP-linearization} follows from the spherical spectral decomposition: the normalized HLS operator has multipliers $0<\vartheta_{i,j}\le1$, while $\mathscr A^{-1}$ has multipliers $0<\nu_{i,j}^{-1}\le1$. Thus both operators are bounded on $L^2(\mathbb S^{2n+1})$, and the uniform Taylor remainders above are $O(\epsilon^2)$ in the dual energy norm. By the Funk--Hecke diagonalization,
\[
 \mathcal B_\mu(\phi,h)=\vartheta_2
 \int_{\mathbb S^{2n+1}}\phi h\,d\sigma,
\]
so \eqref{CP-ell2} gives $L\phi=\ell_2\phi$. Since $\mathscr A^{-1}\phi=\nu_2^{-1}\phi$, \eqref{CP-linearization} and \eqref{CP-M2} imply
\begin{equation}\label{CP-residual-expansion}
 \|\mathscr R_\mu(v_\epsilon)\|_{\mathscr A^{-1}}^2
 =\frac{\ell_2^2}{\nu_2}M_2\epsilon^2+o(\epsilon^2).
\end{equation}
Combining \eqref{CP-distance} and \eqref{CP-residual-expansion}, and using $\ell_2>0$, we find
\begin{equation}\label{CP-quotient-limit}
 \frac{\|\mathscr R_\mu(v_\epsilon)\|_{\mathscr A^{-1}}}
 {\operatorname{dist}_{\mathscr A}(v_\epsilon,\mathfrak M_{1,\mathbb S})}
 =\frac{\ell_2}{\nu_2}+o(1)
 =H_{NS}^{\rm spec}+o(1).
\end{equation}

It remains to pass back to the Heisenberg group and to verify the admissibility condition in the definition of $H_{CP}(1)$. Let $\mathcal T_\mu$ denote the inverse Cayley transform together with the fixed amplitude normalization, so that $\mathcal T_\mu 1=U_\mu$. The conformal covariance identities used to obtain $\mathscr A$ and $\mathcal B_\mu$ imply that there is a constant $c_\mu>0$, independent of $v$, such that
\[
 \|\mathcal T_\mu w\|_{S^{1,2}(\mathbb H^n)}^2
 =c_\mu\langle w,\mathscr Aw\rangle
\]
and, at the level of the weak Euler--Lagrange residual,
\[
 \big\langle R_{\mathbb H}(\mathcal T_\mu v),\mathcal T_\mu h\big\rangle
 =c_\mu\big\langle\mathscr R_\mu(v),h\big\rangle.
\]
Consequently,
\[
 \|R_{\mathbb H}(\mathcal T_\mu v)\|_{(S^{1,2})^{-1}}
 =c_\mu^{1/2}\|\mathscr R_\mu(v)\|_{\mathscr A^{-1}},
 \qquad
 \operatorname{dist}(\mathcal T_\mu v,\mathfrak M_1)
 =c_\mu^{1/2}\operatorname{dist}_{\mathscr A}(v,\mathfrak M_{1,\mathbb S}),
\]
so the residual-to-distance quotient is unchanged.

Set $u_\epsilon:=\mathcal T_\mu v_\epsilon$. For $|\epsilon|$ sufficiently small, $v_\epsilon>0$, hence $u_\epsilon\ge0$, and \eqref{CP-distance} shows that $u_\epsilon\notin\mathfrak M_1$ when $\epsilon\ne0$. Furthermore, $u_\epsilon\to U_\mu$ in $S^{1,2}(\mathbb H^n)$. Testing \eqref{limit3} with $U_\mu$ and using equality in \eqref{eq:HLS'} give
\[
 \|U_\mu\|_{S^{1,2}}^2
 =S_{HL}(Q,\mu)^{\frac{q}{q-1}}
 =S_{HL}(Q,\mu)^{\frac{2Q-\mu}{Q+2-\mu}}.
\]
Thus, for all sufficiently small nonzero $\epsilon$, $u_\epsilon$ satisfies the energy window \eqref{addbound} with $m=1$. Applying the definition of $H_{CP}(1)$ to $u_\epsilon$ and then letting $\epsilon\to0$ in \eqref{CP-quotient-limit} yields
\[
 H_{CP}(1)\le H_{NS}^{\rm spec},
\]
which is \eqref{CP-nonstrict}.
\end{proof}
\endgroup

\subsection*{Acknowledgments}
{This work was supported by the Natural Science Foundation of Chongqing, China} (CSTB2024NSCQ-LZX0038).

\subsection*{Conflict of interest}
On behalf of all authors, the corresponding author states that there is no conflict of interest.

\vspace{.2cm}

\subsection*{Data availability}
{Data sharing is not applicable as no datasets were generated or analyzed in this study.}

\section{Appendix}\label{App}

\begingroup
\subsection{\texorpdfstring{{An explicit sign check at $\mu=2$}}{An explicit sign check at mu=2}}
\begin{lemma}\label{Gamma-mu-two}
For every integer $n\ge1$, the fourth-order coefficient in \eqref{Gamma-def} satisfies
\[
 \Gamma_{n,2}<0.
\]
\end{lemma}
\begin{proof}
Set $\mu=2$. Then
\[
 q=2+\frac1n,
 \qquad
 \vartheta_2=\frac{3}{(2n+1)(2n+3)},
 \qquad
 \vartheta_{11}=\frac{1}{(2n+1)^2},
\]
and
\[
 \vartheta_{22}=\frac{9}{(2n+1)^2(2n+3)^2},
 \qquad
 \vartheta_{40}=\frac{105}{(2n+1)(2n+3)(2n+5)(2n+7)}.
\]
Together with \eqref{higher-eigenvalues}, these identities give
\begin{equation}\label{Gamma2-b}
 b=\frac{2n^2+5n+6}{n(2n+3)}.
\end{equation}
The coefficients in \eqref{dc'-def} become
\begin{equation}\label{Gamma2-c}
\begin{aligned}
 c_{40}&=\frac{2(n+1)(n+6)}{n^2(2n+3)}
 \frac{4n^2+18n+35}{(2n+5)(2n+7)},\\
 c_{11}&=\frac{2(n+1)(2n^2+11n+6)}
 {n^2(2n+1)(2n+3)},\\
 c_{22}&=\frac{2(n+1)(4n^3+26n^2+39n+18)}
 {n^2(2n+1)(2n+3)^2}.
\end{aligned}
\end{equation}
while \eqref{dc-def} yields
\begin{equation}\label{Gamma2-d}
\begin{aligned}
 d_{40}&=\frac{4(n+1)}{n(n+4)(2n+3)}
 \frac{8n^3+63n^2+181n+210}{(2n+5)(2n+7)},\\
 d_{11}&=\frac{4(n+1)^2(5n+6)}
 {n^2(n+4)(2n+1)(2n+3)},\\
 d_{22}&=\frac{4(n+1)(8n^3+31n^2+45n+18)}
 {n^2(2n+1)(2n+3)^2}.
\end{aligned}
\end{equation}
In particular, $d_{40},d_{11},d_{22}>0$ for every $n\ge1$, as is also required by \eqref{dc-def}.

For completeness, substituting \eqref{projection-masses} and the above values into \eqref{C0-def} gives
\begin{equation}\label{Gamma2-C0}
 \mathscr C_0=
 \frac{R_8(n)}
 {n^3(n+1)(n+2)^2(n+3)(n+4)(2n+1)(2n+3)^2(2n+5)(2n+7)},
\end{equation}
where
\begin{align*}
 R_8(n)={}&32n^8+448n^7+2680n^6+8824n^5+18160n^4\\
 &+25594n^3+25325n^2+15516n+3780.
\end{align*}
Now insert \eqref{Gamma2-c}--\eqref{Gamma2-C0}, together with the masses in \eqref{projection-masses}, into
\[
 \Gamma_{n,2}=\mathscr C_0-
 \frac{c_{40}^2}{4d_{40}}m_{40}
 -\frac{c_{11}^2}{4d_{11}}m_{11}
 -\frac{c_{22}^2}{4d_{22}}m_{22}.
\]
Putting the four terms over a common denominator and cancelling common factors gives the exact identity
\begin{equation}\label{Gamma-mu-two-formula}
 \Gamma_{n,2}=-\frac{P_{13}(n)}{D_{13}(n)},
\end{equation}
where
\begin{align*}
P_{13}(n)={}&512n^{13}+13824n^{12}+179704n^{11}+1448280n^{10}
 +7942722n^9+30988200n^8\\
&+87889138n^7+182455437n^6+275750762n^5+297991281n^4\\
&+222632064n^3+108579096n^2+31085856n+3977424,
\end{align*}
and
\begin{align*}
D_{13}(n)={}&2n^2(n+1)^2(n+2)^2(n+3)(n+4)(2n+3)^2(5n+6)\\
&\times(8n^3+31n^2+45n+18)(8n^3+63n^2+181n+210).
\end{align*}
All coefficients of $P_{13}$ are strictly positive, and every factor in $D_{13}$ is strictly positive for $n\ge1$. Hence \eqref{Gamma-mu-two-formula} proves $\Gamma_{n,2}<0$ for every integer $n\ge1$.
\end{proof}
\endgroup

\subsection{\texorpdfstring{{The large-$\mu$ threshold}}{The large-mu threshold}}
\begin{lemma}\label{num}
For $t\geq1$, the function
\begin{equation*}
f(t)=2-\frac{2}{t+4}-2^{\frac{t}{t+1}}
\end{equation*}
admits a unique zero $t_1\approx {4.624}$. Furthermore, $f(t)>0$ for all $t<t_1$ and $f(t)<0$ for all $t>t_1$.
\end{lemma}
\begin{proof}
For $t\geq1$, let
\begin{equation*}
f(t)=2-\frac{2}{t+4}-2^{\frac{t}{t+1}}.
\end{equation*}
 {First,}
\begin{equation}\label{A1}
 f(1)=\frac{8}{5}-2^{\frac{1}{2}}>0,\quad
 \lim\limits_{t\rightarrow+\infty}f(t)=0^-.
\end{equation}
By direct calculations, we have
\begin{equation*}
 f'(t)=\frac{2}{(t+4)^2}-2^{\frac{t}{t+1}}\frac{1}{(t+1)^2}\log 2=\frac{2^{\frac{t}{t+1}}}{(t+4)^2}\left[2^{\frac{1}{t+1}}-\bigg(\frac{t+4}{t+1}\bigg)^2\log 2\right].
\end{equation*}
For $t\geq1$, denote
\begin{equation*}
g(t)=2^{\frac{1}{t+1}}-\bigg(\frac{t+4}{t+1}\bigg)^2\log 2.
\end{equation*}
 Then
\begin{equation}\label{A2}
 g(1)=2^{\frac{1}{2}}-\frac{25}{4}\log 2<2-6\log 2=\log \frac{e^2}{64}<0,\quad \lim\limits_{t\rightarrow+\infty}g(t)=1-\log2>0,
\end{equation}
and
\begin{equation*}
 g'(t)=-2^{\frac{1}{t+1}}\frac{1}{(t+1)^2}\log 2+2\frac{t+4}{t+1}\frac{3}{(t+1)^2}\log 2=\frac{\log 2}{(t+1)^2}\left[6\frac{t+4}{t+1}-2^{\frac{1}{t+1}}\right].
\end{equation*}
Since for all $t\geq1$,
\begin{equation*}
 6\frac{t+4}{t+1}>6>2^{\frac{1}{2}}\geq2^{\frac{1}{t+1}},
\end{equation*}
{we have $g'(t)>0$, so $g$ is strictly increasing on $[1,\infty)$. By \eqref{A2}, there is a unique $t_0>1$ such that $g(t_0)=0$, with $g<0$ on $[1,t_0)$ and $g>0$ on $(t_0,\infty)$. Hence $f$ decreases on $[1,t_0)$ and increases on $(t_0,\infty)$. Combining this monotonicity with \eqref{A1} yields a unique zero $t_1\in(1,t_0)$. Numerically, $t_1\approx{4.624}$.}
\end{proof}


\addcontentsline{toc}{section}{References}

\begin{thebibliography}{99}
\bibitem{A76} T. Aubin,
Probl\`{e}mes isop\'{e}rim\'{e}triques et espaces de Sobolev,
{\em J. Differential Geometry}, {\bf 11} (1976): 573-598.



\bibitem{BE91} G. Bianchi, H. Egnell,
A note on the Sobolev inequality,
{\em J. Funct. Anal.}, {\bf 100} (1991): 18-24.


\bibitem{BL85} H. Brezis, E.H. Lieb,
Sobolev inequalities with remainder terms,
{\em J. Funct. Anal.}, {\bf 62} (1985): 73-86.

\bibitem{C17} E.A. Carlen,
Duality and stability for functional inequalities, {\em Ann. Fac. Sci. Toulouse Math.},  {\bf 26}
(2017): 319-350.

\bibitem{CGK25} S. Chakraborty, M. Ghosh, D. Karmakar,
Existence of Bianchi-Egnell stability extremizer for the Hardy-Sobolev inequality,
{\em preprint}, (2025): arXiv:2505.07039.



\bibitem{CFW13} S.B. Chen, R.L. Frank, T. Weth,
Remainder terms in the fractional Sobolev inequality,
{\em Indiana Univ. Math. J.}, {\bf 62} (2013): 1381-1397.


\bibitem{CLT24'} L. Chen, G.Z. Lu, H.L. Tang,
Stability of Hardy-Littlewood-Sobolev inequalities with explicit lower bounds, {\em Adv. Math.}, {\bf 450} (2024): 109778.

\bibitem{CLT24} L. Chen, G.Z. Lu, H.L. Tang,
Optimal stability of Hardy-Littlewood-Sobolev and Sobolev inequalities of arbitrary orders with dimension-dependent constants,  {\em Math. Ann.}, {\bf 394} (2026): 77.



\bibitem{CLT25} L. Chen, G.Z. Lu, H.L. Tang,
Optimal asymptotic lower bound for stability of fractional Sobolev inequality and the global stability of log-Sobolev inequality on the sphere, {\em Adv. Math.}, {\bf 479} (2025): 110438.

\bibitem{CLTW25} L. Chen, G.Z. Lu, H.L. Tang, B.H. Wang,
Asymptotically sharp stability of Sobolev inequalities on the Heisenberg group
with dimension-dependent constants, {\em J. Math. Pures Appl.}, {\bf 206} (2026): 103832.



\bibitem{CW26} W.J. Chen, Z.X. Wang,
Remainder terms and sharp
quantitative stability for a nonlocal Sobolev inequality on the Heisenberg group,
{\em preprint}, (2026): arXiv:2602.08375.

\bibitem{DK23}
N. De Nitti, T. K\"{o}nig,
Stability with explicit constants of the critical points of the fractional Sobolev inequality and applications to fast diffusion,
{\em J. Funct. Anal.}, {\bf 285} (2023): 110093.



\bibitem{DLYZ25} S.B. Deng, W.S. Luo, M.B. Yang, {X.Y. Zhang},
Some critical fractional Hartree equations: nondegeneracy of the positive solutions and its applications, {\em Sci. China Math.}, {\bf 68} (2025): 2651-2674.

\bibitem{DT23} S.B. Deng, X.L. Tian,
On the stability of Caffarelli-Kohn-Nirenberg inequality in $\mathbb{R}^2$,
{\em preprint}, (2023): arXiv:2308.04111.


\bibitem{DT25} S.B. Deng, X.L. Tian,
On the stability of a nonlocal Sobolev inequality: existence of minimizers,
{\em preprint}, (2025).

\bibitem{DTW25} S.B. Deng, X.L. Tian, J.C. Wei,
Existence of minimizers of the stability constant of Caffarelli-Kohn-Nirenberg inequality,
{\em Proc. Roy. Soc. Edinburgh Sect. A}, (2025), published online first,
DOI: https://doi.org/10.1017/prm.2025.10050.


\bibitem{DTYZ23} S.B. Deng, X.L. Tian, M.B. Yang, S.N. Zhao,
Remainder terms of a nonlocal Sobolev inequality,
{\em Math. Nachr.}, {\bf 297} (2024): 1652-1667.



\bibitem{DEFF25}
J. Dolbeault, M.J. Esteban, A. Figalli, R.L. Frank,
Sharp stability for Sobolev and log-Sobolev inequalities, with optimal dimensional dependence,
{\em Camb. J. Math.}, {\bf 13} (2025): 359-430.


\bibitem{DY19} L.L. Du, M.B. Yang,
Uniqueness and nondegeneracy of solutions for a critical nonlocal equation,
{\em Discrete Contin. Dyn. Syst.}, {\bf 39} (2019): 5847-5866.

\bibitem{FN19} A. Figalli, R. Neumayer,
Gradient stability for the Sobolev inequality: the case $p\geq2$, {\em J. Eur. Math. Soc.},
{\bf 21} (2019): 319-354.

\bibitem{FZ22} A. Figalli, Y.R.Y. Zhang,
Sharp gradient stability for the Sobolev inequality,
{\em Duke Math. J.}, {\bf 171} (2022): 2407-2459.

\bibitem{FS74} G.B. Folland, E.M. Stein,
Estimates for the $\bar{\partial}_b$ complex and analysis on the Heisenberg group,
{\em Comm. Pure Appl. Math.}, {\bf 27}
(1974): 429-522.

\bibitem{FL12} R.L. Frank, E.H. Lieb,
Sharp constants in several inequalities on the Heisenberg group,
{\em Ann. of Math.}, {\bf 176} (2012): 349-381.

\bibitem{GS20} D. Goel, K. Sreenadh,
 Existence and nonexistence results for Kohn Laplacian with Hardy-Littlewood-Sobolev critical exponents, {\em J. Math. Anal. Appl.}, {\bf 486} (2020): 123915.


\bibitem{GHPS19} L. Guo, T.X. Hu, S.J. Peng, W. Shuai,
Existence and uniqueness of solutions for Choquard equation involving Hardy-Littlewood-Sobolev critical exponent, {\em Calc. Var. Partial Differential Equations}, {\bf 58} (2019): 128.

\bibitem{HL28} G.H. Hardy, J.E. Littlewood,
Some properties of fractional integrals, {\em Math. Z.}, {\bf 27} (1928): 565-606.

\bibitem{JL88} D. Jerison, J.M. Lee,
Extremals for the Sobolev inequality on the Heisenberg group and the CR Yamabe problem,
{\em J. Amer. Math. Soc.}, {\bf 1} (1988): 1-13.

\bibitem{K23}
T. K\"{o}nig, On the sharp constant in the Bianchi-Egnell stability inequality,
{\em Bull. Lond. Math. Soc.}, {\bf 55} (2023): 2070-2075.



\bibitem{K25}
T. K\"{o}nig, Stability for the Sobolev inequality: existence of a minimizer,
{\em J. Eur. Math. Soc.}, (2025): DOI 10.4171/JEMS/1582.


\bibitem{LLTX23} X.M. Li, C.X. Liu, X.D. Tang, G.X. Xu,
The nondegeneracy of positive bubble solutions for generalized
energy-critical Hartree equations, {\em Sci. China Math.},
(2026): https://doi.org/10.1007/s11425-025-2549-6.



\bibitem{Li26} J.G. Li,
Critical Geller equations on complex hyperbolic space: nondegeneracy, geodesic symmetry, and global compactness,
{\em preprint}, (2026): arXiv:2608.07240.


\bibitem{L83}
E.H. Lieb,
Sharp constants in the Hardy-Littlewood-Sobolev and related inequalities, {\em Ann. of Math.}, {\bf 118} (1983): 349-374.


\bibitem{LL01}  E.H. Lieb, M. Loss,
{\em Analysis. Second edition}, Graduate Studies in Mathematics, 14, American Mathematical Society,
 Providence, RI, (2001).

 \bibitem{L85} P.L. Lions,
  The concentration-compactness principle in the calculus of variations. The limit case. II,
  {\em Rev. Mat. Iberoamericana}, {\bf 1} (1985): 45-121.


\bibitem{LZ15} H.P. Liu, A. Zhang,
Remainder terms for several inequalities on some groups of Heisenberg-type,
{\em Sci. China Math.}, {\bf 58} (2015): 2565-2580.


\bibitem{L05} A. Loiudice,
Improved Sobolev inequalities on the Heisenberg group,
{\em Nonlinear Anal.},
{\bf 62} (2005): 953-962.

\bibitem{LW2000}
G.Z. Lu, J.C. Wei,  On a Sobolev inequality with remainder terms, {\em Proc. Amer. Math. Soc.}, \textbf{128} (2000): 75-84.

\bibitem{LYZ25} Q.K. Lu, M.B. Yang, S.N. Zhao,
Remainder terms, profile decomposition and sharp quantitative stability in the fractional
nonlocal Sobolev-type inequality with $n>2s$, {\em preprint}, (2025): arXiv:2503.06636.

\bibitem{MU02} A. Malchiodi, F. Uguzzoni,
A perturbation result for the Webster scalar curvature problem on the CR sphere,
{\em J. Math. Pures Appl.}, {\bf 81} (2002): 983-997.


\bibitem{N20} R. Neumayer,
A note on strong-form stability for the Sobolev inequality,
{\em Calc. Var. Partial Differential Equations}, {\bf 59} (2020): 25.

\bibitem{P16} I.D. Platis,
Quasiconformal mappings on the Heisenberg group: an overview,
IRMA Lect. Math. Theor. Phys., {\bf 27}, Eur. Math. Soc., Z\"{u}rich, (2016).



\bibitem{T76}  G. Talenti,
Best constant in Sobolev inequality, {\em Ann. Mat. Pura Appl.}, {\bf 110} (1976): 353-372.

\bibitem{TZZ24} Z.W. Tang, B.W. Zhang, Y.C. Zhang,
Existence of a minimizer for the Bianchi-Egnell inequality on the Heisenberg group,
{\em J. Geom. Anal.}, {\bf 34} (2024): 148.


\bibitem{WW24} J.C. Wei, Y.Z. Wu,
Stability of the Caffarelli-Kohn-Nirenberg inequality: the existence of minimizers,
{\em Math. Z.}, {\bf 308} (2024): 64.

\bibitem{YZ251} M.B. Yang, S.J. Zhang,
Nondegeneracy of positive solutions for critical Hartree equation on Heisenberg group and its applications,
{\em preprint}, (2025): arXiv:2508.07719.

\bibitem{YZ25} W.W. Ye, X.Y. Zhang,
On the stability of a version of nonlocal Sobolev inequality, {\em Bull. Math. Sci.}, {\bf 15} (2025): 2450012.


\bibitem{Z26} Q. Zhang,
Existence of extremals for stability of nonlocal Sobolev inequality,
{\em Adv. Nonlinear Stud.}, {\bf 26} (2026): 274-297.



\bibitem{ZWZLX25}
S.J. Zhang, J.L. Wang, Y. Zheng, X. Li, J.J. Xu,
Symmetry and uniqueness of the positive solution for the critical Hartree equation on the Heisenberg group, {\em preprint}, (2025): arXiv:2511.20264.


\bibitem{ZXW25} S.J. Zhang, J.J. Xu, J.L. Wang,
Quantitative stability of critical points for the nonlocal-Sobolev inequality in Heisenberg group,
{\em preprint}, (2025): arXiv:2508.08614.

\bibitem{ZZZ25} Y.F. Zhang, Y.X. Zhou, W.M. Zou,
Sharp quantitative stability for the fractional Sobolev trace inequality,
{\em Math. Z.}, {\bf 310}
(2025): 90-112.



\end{thebibliography}
\end{document}